\documentclass[11pt]{amsart}
\usepackage{a4wide}

\usepackage{latexsym}
\usepackage{amssymb}
\usepackage{amsmath}
\usepackage{latexsym}
\usepackage[all]{xy}
\usepackage{amscd}
\usepackage{dsfont}
\usepackage{graphics}
\usepackage{graphicx}
\usepackage{bbold}
\usepackage{appendix}

\newtheorem{theorem}{Theorem}[section]

\newtheorem{lemma}[theorem]{Lemma}

\newtheorem{proposition}[theorem]{Proposition}
\newtheorem{corollary}[theorem]{Corollary}

\DeclareMathOperator{\Image}{Im}
\DeclareMathOperator{\Ker}{Ker}

\DeclareMathOperator{\res}{res}
\DeclareMathOperator{\infl}{inf}
\DeclareMathOperator{\tr}{tr}

\DeclareMathOperator{\Opext}{Opext}

\DeclareMathOperator{\Der}{Der}
\DeclareMathOperator{\Inn}{Inn}
\DeclareMathOperator{\Ext}{Ext}

\newcommand{\La}{\mathfrak{a}}
\newcommand{\Lb}{\mathfrak{b}}
\newcommand{\Lc}{\mathfrak{c}}

\newcommand{\Le}{\mathfrak{e}}

\newcommand{\Lg}{\mathfrak{g}}
\newcommand{\Lh}{\mathfrak{h}}

\newcommand{\Ll}{\mathfrak{l}}
\newcommand{\Lm}{\mathfrak{m}}
\newcommand{\Ln}{\mathfrak{n}}
\newcommand{\Lp}{\mathfrak{p}}
\newcommand{\Lq}{\mathfrak{q}}

\begin{document}
\title[Seven-term Exact Sequence for Lie Algebras]{An explicit seven-term exact sequence for the cohomology of a Lie algebra extension}

\author[K. Dekimpe]{Karel Dekimpe}
\address{Katholieke Universiteit Leuven\\
Campus Kortrijk\\
8500 Kortrijk\\
Belgium}
\email{karel.dekimpe@kuleuven-kortrijk.be}
\email{sarah.wauters@kuleuven-kortrijk.be}

\author[M. Hartl]{Manfred Hartl}
\author[S. Wauters]{Sarah Wauters}
\address{Univ. Lille Nord de France\\ F-59000 Lille\\ France \newline UVHC, LAMAV and FR CNRS 2956\\ F-59313 Valenciennes\\ France}
\email{manfred.hartl@univ-valenciennes.fr}

%\author[S. Wauters]{Sarah Wauters}
%\address{Katholieke Universiteit Leuven\\
%Campus Kortrijk\\
%8500 Kortrijk\\
%Belgium}
%\email{sarah.wauters@kuleuven-kortrijk.be}
\thanks{S.\ Wauters is supported by a Ph.~D.~fellowship of the Research Foundation - Flanders (FWO)}
\thanks{Research supported by the Research Fund K.U.Leuven}

\begin{abstract}
We construct a seven-term exact sequence involving low degree cohomology spaces of a Lie algebra $\Lg$, an ideal $\Lh$ of $\Lg$ and the quotient $\Lg / \Lh$ with coefficients in a $\Lg$-module. The existence of such a sequence follows from the Hochschild-Serre spectral sequence associated to the Lie algebra extension. However, some of the maps occurring in this induced sequence are not always explicitly known or easy to describe. In this article, we give alternative maps that yield an exact sequence of the same form, making use of the interpretations of the low-dimensional cohomology spaces in terms of derivations, extensions etc. The maps are constructed using elementary methods. Although we don't know whether the new maps coincide with the ones induced by the spectral sequence, the alternative sequence can certainly be useful, especially since we include straight-forward cocycle descriptions of the constructed maps. 
\end{abstract}

\maketitle
\section{Introduction}
In \cite{hs53-2}, Hochschild and Serre introduced a spectral sequence connecting the cohomology spaces of the Lie algebras in the short exact sequence 
\[\xymatrix{0 \ar[r] & \Lh \ar[r] & \Lg \ar[r] & \Lg / \Lh \ar[r] & 0}\] with coefficients in a $\Lg$-module $M$. This spectral sequence induces a seven-term low degree sequence
\begin{equation}\label{the_sequence}\xymatrix{0 \ar[r] & H^1(\Lg/\Lh,M^{\Lh}) \ar[r]^-{\infl} & H^1(\Lg,M) \ar[r]^-{\res} & H^1(\Lh,M)^{\Lg/\Lh} \ar[r] & H^2(\Lg/\Lh,M^{\Lh}) }\end{equation}
\[\xymatrix{\mbox{} \ar[r]^-{\infl} & H^2(\Lg,M)_1 \ar[r] & H^1(\Lg/\Lh,H^1(\Lh,M)) \ar[r] & H^3(\Lg/\Lh,M^{\Lh})}\] that is exact. However, since the construction of the spectral sequence is complicated in general, it is not always easy to describe the maps occuring in the low-degree sequence. 

Following the work we have done in \cite{dhw11-1} for groups, we will use low-dimensional interpretations of the cohomology spaces to construct maps $\tr: H^1(\Lh,M)^{\Lg/\Lh}\to H^2(\Lg/\Lh,M^{\Lh})$, $\rho:H^2(\Lg,M)_1 \to H^1(\Lg/\Lh,H^1(\Lh,M))$ and $\lambda: H^1(\Lg/\Lh,H^1(\Lh,M))\to H^3(\Lg/\Lh,M^{\Lh})$ that will fit into an exact sequence of the form (\ref{the_sequence}). For the construction, we use elementary concepts from the theory of Lie algebras, such as quotients, semi-direct products etc. We prove exactness of the new sequence without using spectral sequence arguments. Although we don't know whether the maps $\tr$, $\rho$ and $\lambda$ coincide with the maps induced by the spectral sequence, it is clear that these explicit maps can be very useful. In addition, we give a description of the maps on cocycle level in Section \ref{cocycle}. 

In Section \ref{huebschmann_section}, we introduce crossed complexes and crossed extensions, as in \cite{hueb80-1}, in the Lie algebra setting. We then go on to define free crossed complexes, morphisms of crossed complexes and homotopy of such morphisms. We use these objects to give a characterisation of the trivial element in $H^3(\Lg, \Lh, M^{\Lh})$ that is needed when we prove that the sequence is exact at $H^1(\Lg/\Lh, H^1(\Lh,M))$ (see Lemma \ref{Huebschmann} in Section \ref{lambda}). For groups, this characterisation has been proved by Huebschmann in \cite{hueb80-1}. We think that Section \ref{huebschmann_section} can be interesting in its own right.  

There is a clear correspondence with the group case treated in \cite{dhw11-1}. However, for groups we chose to work with an interpretation of the first cohomology group in terms of semi-direct complements, while in this article, we work with derivations. It is not difficult though to restate the approach for groups in terms of derivations.

In Section \ref{der_sec_sdc}, we give a survey of the interpretations of the first and second cohomology group. Section \ref{pb_po} treats the notions of pull-back and push-out construction, two notions that will be used throughout the article. The next section deals with some homological algebra that will be useful to prove that the maps we constructed are group homomorphisms. The maps $\tr$ and $\rho$ are introduced in Sections \ref{tr} and \ref{second_map} respectively. Section \ref{third_cohom} gives an interpretation of the third cohomology group, that is used in Section \ref{lambda} to construct the map $\lambda$. The next section proves that the three maps are natural, and in Section \ref{cocycle} we give cocycle descriptions of the new maps. We then treat an example in Section \ref{heisenberg_example}, to finish with the proof of Lemma \ref{Huebschmann} in Section \ref{huebschmann_section}.

\section{Interpretations of the first and second cohomology group}\label{der_sec_sdc}
We recall the interpretations for the first and second cohomology of a Lie algebra $\Lg$ over a field $k$ with coefficients in a $\Lg$-module $M$. %. 
These interpretations are standard and can be found in many textbooks, such as \cite{weib94-1}. 
An interpretation for the third cohomology group will be given in Section \ref{third_cohom}.

First recall that a $\Lg$-module $M$ is a $k$-module equipped with a $k$-bilinear map $\Lg \otimes_k M \to M$, mapping $x \otimes m$ to $x \cdot m$ and satisfying the relation $[x,y]\cdot m = x \cdot(y \cdot m)-y \cdot (x \cdot m)$ for all $x, \, y \in \Lg$ and $m \in M$.
If $M$ is a $\Lg$-module, a \emph{derivation} $d : \Lg \rightarrow M$ is a $k$-linear map satisfying $d[x,y]= x \cdot d(y) - y \cdot d(x)$ for all $x, \, y\in \Lg$, where $\cdot$ denotes the action of $\Lg$ on $M$. The set of all derivations from $\Lg$ to $M$ is called $\Der (\Lg , M)$. For every $m \in M$, we can define a derivation $d_m : x \mapsto x \cdot m$, and sometimes we use the notation $d_m = \delta (m) $. This is what we call an \emph{inner derivation}, and the set of all inner derivations from $\Lg$ to $M$ is denoted by $\Inn (\Lg,M)$. 
It is a well-known result that $H^1(\Lg,M)$ is isomorphic to $\Der (\Lg,M) / \Inn (\Lg,M)$. One can prove this by abstract methods, or by observing that $\Der (\Lg, M)$ are exactly the 1-cocycles and $\Inn (\Lg, M)$ are exactly the 1-coboundaries arising from the Chevalley-Eilenberg complex of $\Lg$ with coefficients in $M$. 
Like in the group case, there is a correspondence between derivations and splittings  of the standard split extension. If $M$ is a $\Lg$-module, we can construct the standard split extension
\[\xymatrix{\underline{e}_0: &0 \ar[r] & M \ar[r]^-{i_0} & M \rtimes \Lg \ar[r]^-{p_0} & \Lg \ar[r] & 0.}\] A \emph{splitting} $s$ of $\underline{e}_0$ is a Lie algebra map $s : \Lg \rightarrow M \rtimes \Lg$ such that $p_0 \circ s = \mathbb{1}_{\Lg}$. It is clear that for every splitting $s$, there exists a map $d: \Lg \rightarrow M$ such that $s(x)=(d(x),x)$. It follows from the definition of the Lie bracket in the semi-direct product that the fact that $s$ is a Lie algebra homomorphism is equivalent to $d$ being a derivation. This means that there is a one-to-one correspondence between splittings of the standard split extension of $\Lg$ by $M$ and derivations $d : \Lg \rightarrow M$. 
In a large part of this article, we will deal with the cohomology of a Lie algebra $\Lh$ that is an ideal of a Lie algebra $\Lg$. In this case, there is a Lie algebra action of $\Lg / \Lh$ on the cohomology spaces of $\Lh$. It is well-known that the action of $\Lg / \Lh$ on $H^1(\Lh, M)$ is induced by the action of $\Lg$ on $\Der (\Lh,M)$, defined as $({}^{x}d) (z)=x \cdot d(z)-d[x,z]$ for all $x \in \Lg $ and $z \in \Lh$. 

Let $\Lg$ be a Lie algebra and $M$ an abelian Lie algebra. A \emph{Lie algebra extension} of $\Lg$ by $M$ is an exact sequence of Lie algebras
\begin{equation}\label{extension_def}\underline{e}: \xymatrix{0 \ar[r] & M \ar[r]^i & \Le \ar[r]^{p} & \Lg \ar[r] & 0.}\end{equation}
The Lie bracket in $\Le$ induces a Lie algebra action of $\Lg$ on $M$, defined by $i(x \cdot m) = [\widetilde{x},i(m)]$ for $m \in M$ and $x \in \Lg$ with $\widetilde{x}$ a pre-image of $x$ under $p$. The action is well-defined since $M$ is abelian. If we take any $\Lg$-module $M$, we define $\Ext (\Lg , M)$ to be the class of Lie algebra extensions of $\Lg$ by $M$ such that action of $\Lg$ on $M$ induced by the extension coincides with the module structure. Two extensions $\underline{e}$ and $\underline{e}'$ of $\Lg$ by $M$ are \emph{equivalent} if there is a Lie algebra map $h : \Le \rightarrow \Le'$ making the diagram
\[\xymatrix{\underline{e} :  0 \ar[r] & M \ar[r] \ar@{=}[d] & \Le \ar[r] \ar[d] & \Lg \ar[r] \ar@{=}[d] & 0 \\
\underline{e}':   0 \ar[r] & M \ar[r] & \Le' \ar[r] & \Lg \ar[r] & 0}\] commute. It is well-known that $H^2(\Lg,M)$ is isomorphic to $\Ext (\Lg ,M) / \sim$. 
To make the correspondence explicit, let 
\[\xymatrix{0 \ar[r] & M \ar[r]^{i} & \Le \ar[r]^{p} & \Lg \ar[r] & 0}\] be an extension of $\Lg$ by $M$, and choose a section $s : \Lg \rightarrow \Le$, i.e.~a $k$-linear map with $p \circ s = \mathbb{1}_{\Lg}$. Now we define the map $f : \Lg \times \Lg \rightarrow M$, $f(x,y)=[s(x),s(y)]-s[x,y] \in M$ for $x, \, y \in \Lg$.  In a way, $f$ measures the extent in which $s$ is not a Lie algebra map. Observe that $f$ is a bilinear alternating map, giving rise to a linear map $\widetilde{f} : \Lambda^2 \Lg \rightarrow M$. This is exactly the corresponding 2-cocycle coming from the Chevalley-Eilenberg complex. Changing the choice of the section $s$ alters the cocycle with a coboundary. 
Conversely, 
if we are given a 2-cocycle $f : \Lambda^2 \Lg \rightarrow M$, we construct a corresponding extension as follows. Take the vector space $\Le = M \times \Lg$, and make it into a Lie algebra by putting $[(m,x),(m',x')] =(x \cdot m' - x' \cdot m + f(x \wedge x'), [x,x'])$. Since $f$ is a cocycle, this indeed defines a Lie bracket. Furthermore, 
\[\xymatrix{0 \ar[r] & M \ar[r]^{i} & \Le \ar[r]^{p} & \Lg \ar[r] & 0}\] is an extension of $\Lg$ by $M$. If we choose the section $s : \Lg \rightarrow \Le$ mapping $x$ to $(0,x)$, the associated cocycle is indeed $f$.

\section{Pull-back and push-out constructions}\label{pb_po}

We recall two concepts that we will need throughout the article.

\subsection{Pull-back} This is a basic concept from homological algebra that can be found in any textbook on the topic, e.g.~\cite{weib94-1}.
The pull-back $\Lp$ of two Lie algebra homomorphisms $\lambda : \Lb \rightarrow \La$ and $\mu : \Lc \rightarrow \La$ is the sub-algebra of $\Lb \times \Lc$ consisting of all pairs $(x,y) \in \Lb \times \Lc$ with $\lambda(x)=\mu(y)$. There are maps $\overline{\lambda} : \Lp \rightarrow \Lc$ and $\overline{\mu} : \Lp \rightarrow \Lb$ induced by projection, and the following diagram is commutative. 
\[\xymatrix{ \Lp \ar[d]^{\overline{\lambda}} \ar[r]^{\overline{\mu}} & \Lb \ar[d]^{\lambda} \\ \Lc \ar[r]^{\mu} & \La}\]
The pull-back satisfies the following universal property: if $\Lp'$ is a Lie algebra, and $i : \Lp' \rightarrow \Lb$ and $j : \Lp' \rightarrow \Lc$ are Lie algebra homomorphisms such that $\lambda \circ i=\mu \circ j$, then there exists a unique Lie algebra homomorphism $\phi: \Lp' \rightarrow \Lp$ satisfying $\overline{\mu} \circ \phi = i$ and $\overline{\lambda} \circ \phi=j$.

Now let $\xymatrix@1{\underline{e} : 0 \ar[r] & M \ar[r]^{i} & \Le \ar[r]^{\pi} & \Lg \ar[r] & 0}$ be an extension of Lie algebras. If a Lie algebra map $\tau : \Lg' \rightarrow \Lg$ is given, we can take the pull-back $\Lp$ of $\pi$ and $\tau$, and construct a map $\iota : M \rightarrow \Lp$ with $\tau \circ \iota = i$ and $\overline{\pi} \circ \iota$ the trivial map, using the universal property of $\Lp$. Observe that $M$ is a $\Lg'$-module through $\tau$. 
Now the following property holds.
\begin{proposition}\label{pb}
The diagram \[\xymatrix{\underline{e} : 0 \ar[r] & M \ar@{=}[d]\ar[r]^{i} & \Le \ar[r]^{\pi} & \Lg \ar[r] & 0\\\underline{e}': 0 \ar[r] & M \ar[r]^{\iota} & \Lp \ar[u]^{\overline{\tau}} \ar[r]^{\overline{\pi}} & \Lg' \ar[u]^{\tau}\ar[r] & 0}\] is commutative with exact rows. Furthermore, $[\underline{e}']=\tau^*[\underline{e}]$ in $H^2(\Lg', M)$, where $\tau^*$ is the map induced by $\tau$ on cohomology level.
\end{proposition}
\begin{proof}
It is not difficult to show that the bottom row is exact. Furthermore, commutativity is given by the definition of the pull-back and the construction of $\iota$. The fact that $[\underline{e}']=\tau^*[\underline{e}]$ follows from a cocycle argumentation.
\end{proof}
The following proposition is the converse statement.

\begin{proposition}\label{exact_is_pb}
Let \[\xymatrix@1{\underline{e}: 0 \ar[r] & M \ar[r]^i & \Le \ar[r]^{\pi} & \Lg \ar[r] & 0\\ \underline{e}': 0 \ar[r] & M \ar[r]^j\ar@{=}[u] & \Le' \ar[u]^{\sigma} \ar[r]^p & \Lg' \ar[r] \ar[u]^{\tau} & 0}\] be a commutative diagram with exact rows. Then $[\underline{e}']=\tau^*[\underline{e}]$ in $H^2(\Lg', M)$ (so $\Le'$ is isomorphic to the pull-back of $\tau$ and $\pi$). 
\end{proposition}
\begin{proof}
Call $\Lp$ the pull-back of $\tau$ and $\pi$. The previous proposition shows that the diagram 
\[\xymatrix{ 0 \ar[r] & M \ar@{=}[d]\ar[r]^{i} & \Le \ar[r]^{\pi} & \Lg \ar[r] & 0\\0 \ar[r] & M \ar[r]^{\iota} & \Lp \ar[u]^{\overline{\tau}} \ar[r]^{\overline{\pi}} & \Lg' \ar[u]^{\tau}\ar[r] & 0}\] is commutative with exact rows. Using the universal property, we can define a map $\phi : \Le' \rightarrow \Lp$ such that $\overline{\tau} \circ \phi = \sigma$ and $\overline{\pi} \circ \phi=p$. It is immediate that the right hand side of the diagram 
\[\xymatrix{ 0 \ar[r] & M \ar@{=}[d]\ar[r]^j & \Le' \ar[d]^{\phi} \ar[r]^{p} & \Lg' \ar[r] & 0\\ 0 \ar[r] & M \ar[r]^{\iota} & \Lp  \ar[r]^{\overline{\pi}} & \Lg' \ar@{=}[u]\ar[r] & 0}\] commutes. Furthermore, post-composing both $\phi \circ j$ and $\iota$ with $\overline{\pi}$ yields the trivial map, and post-composing both maps with $\overline{\tau}$ yields the injection $i$. By the uniqueness in the universal property of the pull-back, $\phi \circ j =\iota$. The five-lemma now implies that $\phi$ is an isomorphism. The previous lemma shows that $[\underline{e}]=\tau^*[\underline{e}']$.
\end{proof}

\subsection{Push-out construction}
For groups, this construction is described in several articles, for example \cite{cw81-1} and \cite{ratc80-1}. We adapt the definition to the case of Lie algebras. We start with a Lie algebra extension $\xymatrix@1{\underline{e}: 0 \ar[r] & M \ar[r]^{\iota} & \Le \ar[r]^{\pi} & \Lg \ar[r] & 0}$. We already know that $\underline{e}$ gives rise to a Lie algebra action of $\Lg$ on $M$, defined by $\iota(x \cdot y) = [\widetilde{x},\iota(y)]$ for all $x\in \Lg$, $y \in M$ and $\pi(\widetilde{x})=x$. Take another $\Lg$-module $M'$ and a $\Lg$-module map $\alpha : M \rightarrow M'$. Now we can form the semi-direct product $M' \rtimes \Le$, where the action of $\Le$ on $M'$ is given through $\pi$. One can check that the set $S = \{(\alpha(x), -\iota(x)) \ | \ x \in M\}$ is an ideal of $M'\rtimes \Le$. Define $\Ll$ as the quotient algebra $(M' \rtimes \Le) \big/ S$. We call $\Ll$ the \emph{push-out construction of $\iota$ and $\alpha$}, and there are maps $\overline{\iota} : M' \rightarrow \Ll$ and $\overline{\alpha} : \Le \rightarrow \Ll$ induced by inclusion, making the diagram
\[\xymatrix{M \ar[r]^{\iota} \ar[d]^{\alpha} & \Le \ar[d]^{\overline{\alpha}}\\ M'  \ar[r]^{\overline{\iota}}  & \Ll }\]
commute. The push-out construction satisfies the following universal property: if $\Ll'$ is a Lie algebra, and $i : M' \rightarrow \Ll'$ and $j : \Le \rightarrow \Ll'$ are Lie algebra maps satisfying $i \circ \alpha= j \circ \iota$ and the property $i(\pi(x) \cdot y)=[j(x),i(y)]$ holds for all $x \in \Le$ and $y \in M'$, then there exists a unique Lie algebra map $\psi : \Ll \rightarrow \Ll'$ such that $\psi \circ \overline{\alpha}=j$ and $\psi \circ \overline{\iota} = i$. 
If we define a Lie algebra action of $\Le$ on $\Ll'$ by putting $x \cdot y = [j(x), y]$, then the condition $i(x \cdot y)=[j(x),i(y)]$ is equivalent with $i$ being an $\Le$-module map.

Remember we started with a Lie algebra extension $\underline{e}$. Using the universal property for the push-out construction, we can find a map $p : \Ll \rightarrow \Lg$ such that $p \circ \overline{\iota}$ is trivial and $p \circ \overline{\alpha}= \pi$. The following property holds.
\begin{proposition}
The diagram \[\xymatrix{\underline{e} : 0 \ar[r] & M \ar[d]^{\alpha}\ar[r]^{\iota} & \Le \ar[r]^{\pi} \ar[d]^{\overline{\alpha}} & \Lg \ar[r] & 0\\\underline{e}': 0 \ar[r] & M' \ar[r]^{\overline{\iota}} & \Ll \ar[r]^{p} & \Lg \ar@{=}[u]\ar[r] & 0}\] is commutative with exact rows. Furthermore, $[\underline{e}']=\tau_*[\underline{e}]$ in $H^2(\Lg, M')$, where $\tau_*$ is the map induced by $\tau$ on cohomology level.
\end{proposition}
\begin{proof}
It is straight-forward that the bottom sequence is exact. Commutativity at the left hand side is given by the definition of the push-out construction, whereas commutativity at the right hand side comes from the construction of $p$. The fact that $[\underline{e}']=\tau_*[\underline{e}]$ is verified using a cocycle argumentation.
\end{proof}
Like for the pull-back, we also have the converse proposition.
\begin{proposition} \label{exact_is_poc}
If \[\xymatrix{\underline{e} : 0 \ar[r] & M \ar[d]^{\alpha}\ar[r]^{\iota} & \Le \ar[r]^{\pi} \ar[d]^{\sigma} & \Lg \ar[r] & 0\\\underline{e}': 0 \ar[r] & M' \ar[r]^{j} & \Le' \ar[r]^{\pi'} & \Lg \ar@{=}[u]\ar[r] & 0}\]
is a commutative diagram with exact rows, with $M$ and $M'$ having a $\Lg$-module structure induced by $\underline{e}$ and $\underline{e}'$ respectively, and $\alpha$ is an $\Lg$-module map, then $[\underline{e}']=\tau_*[\underline{e}]$ in $H^2(M',\Lg)$ (so $\Le'$ is isomorphic to the push-out construction of $\iota$ and $\alpha$).
\end{proposition}
\begin{proof}
Call $\Ll$ the push-out construction of $\iota$ and $\alpha$. Using the universal property, we construct a map $\psi : \Ll \rightarrow \Le '$ with $\psi \circ \overline{\alpha} = \sigma$ and $\psi \circ \overline{\iota} = j$. (The condition $j(\pi(x) \cdot y)=[\sigma(x), j(y)]$ is satisfied since $\pi(x)=\pi' \circ \sigma (x)$.) We claim that the diagram 
\[\xymatrix{ 0 \ar[r] & M' \ar@{=}[d]\ar[r]^{\overline{\iota}} & \Ll \ar[d]^{\psi} \ar[r]^{p} & \Lg \ar[r] & 0\\ 0 \ar[r] & M' \ar[r]^{j} & \Le'  \ar[r]^{\pi'} & \Lg \ar@{=}[u]\ar[r] & 0}\]
is commutative. It follows from the definition of $\psi$ that the left hand side is commutative. Pre-composing both $p$ and $\pi' \circ \psi$ with $\overline{\iota}$ gives the trivial map, and pre-composing both maps with $\overline{\alpha}$ gives $\pi$. Now by uniqueness in the universal property, we conclude that $p = \pi ' \circ \psi$. By the previous proposition, both rows of the diagram are exact, and the five-lemma shows that $\psi$ is an isomorphism.  It follows that $[\underline{e}']=\tau_*[\underline{e}]$.
\end{proof}

\section{A categorical point of view}
A category $\mathcal{C}$ \emph{admits group objects} if it has finite products and a final object $T$. A \emph{group object} in $\mathcal{C}$ is an object $X$ with morphisms $\mu : X \times X \rightarrow X$ (``group law''), $\eta : T \rightarrow X$ (``neutral element'') and $i: X \rightarrow X$ (``inverse element''), satisfying axioms similar to the group axioms (e.g.\ see \cite[page 75]{macl71-1}). We denote the group object by $(X, \mu, \eta, i)$ or simply by $X$.

The following lemma is well-known.
\begin{lemma}
Let $F : \mathcal{C} \rightarrow \mathcal{D}$ be a functor between two categories that admit group objects. If $F$ preserves products and the final object, then $F$ preserves group objects.
\end{lemma}
More precise, if $(X,\mu, \eta, i)$ is a group object in $\mathcal{C}$, $(FX, F(\mu) \circ j^{-1}, F(\eta), F(i))$ is a group object in $\mathcal{D}$, where $j : F(X \times X) \rightarrow FX \times FX$ is the isomorphism coming from the fact that $F$ preserves products. 

As a result, we can prove the following lemma.
\begin{lemma}\label{they_preserve}
Both $H^n(\Lg, -)$ and $\Der(\Lg,-) : \Lg\mbox{-Mod} \rightarrow k\mbox{-Mod}$ preserve group objects. Moreover, the group structure of $H^n(\Lg, M)$ and $\Der (\Lg,M)$ induced by the abelian group law of a $\Lg$-module $M$ is the standard group structure.
\end{lemma}

The next lemma is also a well-known fact.
\begin{lemma}\label{homomorphism_of_group_objects}
If $\mathcal{C}$ and $\mathcal{D}$ are categories admitting group objects, and $a : F \rightarrow G$ is a natural transformation between two functors $F, \, G : \mathcal{C} \rightarrow \mathcal{D}$ that preserve group objects, then $a_X : FX \rightarrow GX$ is a group homomorphism for every object $X$ in $\mathcal{C}$.
\end{lemma}
Here, a morphism $m : X \rightarrow Y$ of $\mathcal{C}$ between two group objects $(X, \mu_X, \eta_X, i_X)$ and $(Y,\mu_Y,\eta_Y, i_Y)$ is a \emph{group homomorphism} if the  diagram 
\[\xymatrix{X \times X \ar[r]^-{\mu_X} \ar[d]^{m \times m} & X \ar[d]^{m}\\
Y \times Y \ar[r]^-{\mu_Y} & Y}\]
is commutative. 
This lemma will be a powerful tool in showing that the maps we construct are group homomorphisms in the traditional sense.

\section{The map $tr$}\label{tr}

Given a short exact sequence of Lie algebras over a fixed field $k$
\[\xymatrix{0 \ar[r] & \mathfrak{h} \ar[r] & \mathfrak{g} \ar[r] & \mathfrak{g}/\mathfrak{h} \ar[r] & 0,}\] and a $\Lg$-module $M$. There is a seven-term exact sequence  induced by the Hochschild-Serre spectral sequence. We will construct a seven-term exact sequence of the same form
\[\xymatrix{0 \ar[r] & H^1(\Lg/\Lh,M^{\Lh}) \ar[r]^-{\infl} & H^1(\Lg,M) \ar[r]^-{\res} & H^1(\Lh,M)^{\Lg/\Lh} \ar[r]^-{\tr} & H^2(\Lg/\Lh,M^{\Lh}) }\]
\[\xymatrix{\mbox{} \ar[r]^-{\infl} & H^2(\Lg,M)_1 \ar[r]^-{\rho} & H^1(\Lg/\Lh,H^1(\Lh,M)) \ar[r]^-{\lambda} & H^3(\Lg/\Lh,M^{\Lh})},\] where  $H^2(\Lg,M)_1$ denotes the kernel of the restriction map $ \res:H^2(\Lg,M) \to H^2(\Lh,M)$. To do this, we will explicitly construct maps $\tr$, $\rho$ and $\lambda$, using the interpretations of the cohomology spaces, such that we get an exact sequence as above. First, we construct an appropriate map $\tr : H^1(\Lh,M)^{\Lg/\Lh} \to  H^2(\Lg/\Lh,M^{\Lh})$.

We will introduce the desired map in a more general setting. Take an extension of Lie algebras \begin{equation}\label{eq:seq}\underline{e}: \xymatrix{0 \ar[r] & M \ar[r]^{i} & \mathfrak{e} \ar[r]^{p} & \mathfrak{g} \ar[r] & 0,}\end{equation}
for which there exists a Lie algebra homomorphism $s : \mathfrak{h} \rightarrow \mathfrak{e}$ with $p \circ s = \mbox{id}_{\mathfrak{h}}$. We will call $\underline{e}$ a \emph{partially split extension} and $s$ a \emph{partial splitting} of $\underline{e}$. 
%Next, we show that we can associate a derivation to every $x \in \mathfrak{e}$. 
For each $x \in \mathfrak{e}$, define the map \[d^{s}_x : \Lh \rightarrow M; z \mapsto [x,s(z)] -s[p(x),z].\] 
%This will be an important notion throughout this section and the next one. 
The map $d^s_x$ is linear, and obviously depends on $s$ (see also Lemma \ref{change_section}). The following observation is immediate.
\begin{lemma}
For every $x \in \mathfrak{e}$, the map $d^s_x : \mathfrak{h} \rightarrow M$ is a derivation. 
\end{lemma}
Observe that $d^s_x$ measures the defect of $s$ preserving the action of $x \in \mathfrak{e}$, where the $\mathfrak{e}$-action on $\mathfrak{e}$ is given by the left adjoint action, and the action on $\mathfrak{g}$ is the one induced by $p : \mathfrak{e} \rightarrow \mathfrak{g}$.  

Define $I_{\mathfrak{e}}(s(\mathfrak{h}))$ as the subalgebra of $\mathfrak{e}$ consisting of all elements $x \in \mathfrak{e}$ such that $[x,s(\mathfrak{h})] \subset s(\mathfrak{h})$. This is the largest Lie subalgebra of $\Le$ in which $\Image s$ is an ideal. The proof of the following lemmas is left to the reader.

\begin{lemma}
The intersection $i(M) \cap I_{\mathfrak{e}}(s(\mathfrak{h}))$ equals $i(M^{\mathfrak{h}})$.
\end{lemma}

\begin{lemma}\label{surjective}
The map $p|_{I_{\mathfrak{e}}(s(\mathfrak{h}))} : I_{\mathfrak{e}}(s(\mathfrak{h})) \rightarrow \mathfrak{g}$ is surjective iff $d^{s}_x$ is an inner derivation for all $x \in \mathfrak{e}$.
\end{lemma}

Let $\Omega$ be the set of all pairs $(\underline{e},s)$, where $\underline{e}$ is a partially split extension of the form (\ref{eq:seq}) and $s : \mathfrak{h} \rightarrow M$ is a partial section such that $d^s_x$ is an inner derivation for all $x \in \mathfrak{e}$. The preceding lemmas show that in this case, the sequence 
\[\xymatrix{0 \ar[r]& M^{\mathfrak{h}} \ar[r]^-{i} & I_{\mathfrak{e}}(s(\mathfrak{h})) \ar[r]^-{p} & \mathfrak{g} \ar[r] & 0}\] is exact. It is then immediate that the sequence 
\[\underline{e}':\xymatrix{0 \ar[r]& M^{\mathfrak{h}} \ar[r]^-{\overline{i}} & I_{\mathfrak{e}}(s(\mathfrak{h}))/s(\mathfrak{h}) \ar[r]^-{\overline{p}} & \mathfrak{g}/\mathfrak{h} \ar[r] & 0}\]
is also exact. We can thus define a map 
\[\omega : \Omega \rightarrow H^2(\mathfrak{g}/\mathfrak{h},M^{\mathfrak{h}}),\]
mapping the pair $(\underline{e},s)$ to the equivalence class of $\underline{e}'$. 

We first make some remarks about the construction of $\omega$. In the commutative diagram
\[\xymatrix{0 \ar[r] & M \ar[r]^i  & \mathfrak{e} \ar[r]^p & \mathfrak{g} \ar[r] & 0\\
0 \ar[r] & M^{\mathfrak{h}} \ar@{^{(}->}[u]^{j} \ar@{=}[d] \ar[r]^{i \ \ } & I_{\mathfrak{e}}(s(\mathfrak{h})) \ar[u]^{\iota} \ar[d] \ar[r]^{\ \ p} & \mathfrak{g} \ar@{=}[u] \ar[d]^{\pi}\ar[r] & 0\\
0 \ar[r] & M^{\mathfrak{h}} \ar[r]^{\overline{i} \ \ \ } & I_{\mathfrak{e}}(s(\mathfrak{h}))/s(\mathfrak{h}) \ar[r]^{\ \ \ \overline{p}}& \mathfrak{g}/\mathfrak{h} \ar[r] & 0,}\]
$\mathfrak{e}$ is the push-out construction of the inclusion map $j:M^{\mathfrak{h}} \hookrightarrow M$ and $i:M^{\mathfrak{h}} \rightarrow I_{\mathfrak{e}}(s(\mathfrak{h}))$, and $I_{\mathfrak{e}}(s(\mathfrak{h}))$ is the pull-back of $\overline{p}$ and $\pi $. This follows from Proposition \ref{exact_is_pb} and Proposition \ref{exact_is_poc}. 

Now we restrict ourselves to the case where the Lie algebra extension (\ref{eq:seq}) is the standard split extension 
\[\underline{e}_0: \xymatrix{0 \ar[r] & M \ar[r] & M \rtimes \mathfrak{g} \ar[r] & \mathfrak{g} \ar[r]& 0.}\] 
Take a derivation $d: \Lh \rightarrow M$. The map $s_d : \Lh \rightarrow M \rtimes \Lg$ defined as $s_d(z)=(d(z),z)$ for $z \in \mathfrak{h}$ is a partial splitting of $\underline{e}_0$, and the derivation $d^{\, s_d}_{(m,x)}$ equals $^{x} d - \delta(m)$. It follows that $(\underline{e}_0,s_d) \in \Omega$ iff $^{x} d$ is an inner derivation for all $x \in \mathfrak{e}$, or equivalently, $[d] \in H^1(\mathfrak{h},M)^{\mathfrak{g}/\mathfrak{h}}$. If we denote by $\mbox{Der}(\mathfrak{h},M)^{\mathfrak{g}/\mathfrak{h}}$ the pre-image of $H^1(\mathfrak{h},M)^{\mathfrak{g}/\mathfrak{h}}$ under the projection map $\mbox{Der}(\mathfrak{h},M) \rightarrow H^1(\mathfrak{h},M)$, we can define a map 
$\hat{\tr}: \mbox{Der}(\mathfrak{h},M)^{\mathfrak{g}/\mathfrak{h}} \rightarrow H^2(\mathfrak{g}/\mathfrak{h},M^{\mathfrak{h}})$ by
\[\hat{\tr}(d)=\omega(\underline{e}_0,s_d).\] It turns out that this is the appropriate map.

\begin{lemma}\label{tr_natural}
The map $\hat{\tr}: \mbox{Der}(\mathfrak{h},-)^{\mathfrak{g}/\mathfrak{h}} \rightarrow H^2(\mathfrak{g}/\mathfrak{h},-^{\mathfrak{h}})$ is a natural transformation of functors.
\end{lemma}
\begin{proof}
We have to show that for every $\mathfrak{g}$-module map $\alpha : M_1 \rightarrow M_2$, the following diagram is commutative.
\[\xymatrix{\mbox{Der}(\mathfrak{h},M_1)^{\mathfrak{g}/\mathfrak{h}} \ar[r]^{\hat{\tr}} \ar[d]^{\alpha_*} & H^2(\mathfrak{g}/\mathfrak{h},M_1^{\mathfrak{h}}) \ar[d]^{\alpha_*}\\
\mbox{Der}(\mathfrak{h},M_2)^{\mathfrak{g}/\mathfrak{h}} \ar[r]^{\hat{\tr}}  & H^2(\mathfrak{g}/\mathfrak{h},M_2^{\mathfrak{h}}) }\]
Take $d \in \mbox{Der}(\mathfrak{h},M_1)^{\mathfrak{g}/\mathfrak{h}}$ and construct an extension representing $\hat{\tr}(d)$. Use the push-out construction with respect to the restriction $M_1^{\Lh} \to M_2^{\Lh}$ of $\alpha$ to find $\alpha_*(\hat{\tr}(d))$. On the other hand, we can construct an extension representing $\hat{\tr}(\alpha_*(d))$, with $\alpha_*(d)=\alpha \circ d$. Now the map $\alpha \rtimes \mathbb{1} : M_1 \rtimes \Lh \to M_2 \rtimes \Lh$ induces an equivalence between the two extensions, by the universal property of the push-out, the five-lemma and the Proposition \ref{exact_is_pb} and Proposition \ref{exact_is_poc}. (Lemma \ref{char_I} can be useful to characterise the extension algebras.)
\end{proof}

Now the following corollary follows directly from Lemma \ref{they_preserve} and Lemma \ref{homomorphism_of_group_objects}.
\begin{corollary}\label{tr_group_hom}
The map $\hat{\tr}: \mbox{Der}(\mathfrak{h},-)^{\mathfrak{g}/\mathfrak{h}} \rightarrow H^2(\mathfrak{g}/\mathfrak{h},-^{\mathfrak{h}})$ is a homomorphism of groups.
\end{corollary}

\begin{lemma}\label{extend_derivation}
For $d \in \mbox{Der}(\mathfrak{h},M)^{\mathfrak{g}/\mathfrak{h}}$, $\hat{\tr}(d) \equiv 0$ iff there exists a derivation $\widetilde{d} : \mathfrak{g} \rightarrow M$ extending $d$.
\end{lemma}
\begin{proof}
Put $I=I_{M \rtimes \Lg}(s_d(\Lh))$.
If $d=\widetilde{d}|_{\mathfrak{h}}$, we get a splitting $\widetilde{s}$ extending $s_d$, defined by $\widetilde{s}(x)=(\widetilde{d}(x),x)$.  Obviously $s_d(\mathfrak{h})$ is an ideal in $\widetilde{s}(\mathfrak{g})$, so $\widetilde{s}$ has image in $I$. It follows immediately that the quotient map $\overline{s} : \mathfrak{g}/\mathfrak{h} \rightarrow I / s_d(\mathfrak{h})$ is a well-defined splitting of the sequence $\underline{e}'$, representing $\hat{\tr}(d)$, so $\hat{\tr}(d)\equiv 0$.
Conversely, suppose that $\hat{\tr}(d) \equiv 0$. This means that there is a splitting $\alpha : \mathfrak{g}/\mathfrak{h} \rightarrow I / s_d(\mathfrak{h})$ of $\underline{e}'$. 
Since $I$ is the pull-back of $\overline{p}$ and $\pi$, we can use the pull-back property to construct a splitting $\alpha' : \mathfrak{g} \rightarrow I$ of $p$, such that the composition with the quotient map $I \rightarrow I/ s_d(\mathfrak{h})$ equals $\alpha \circ \pi$. It follows that $\alpha'(\mathfrak{h}) \subset s_d(\mathfrak{h})$, so $\alpha'|_{\mathfrak{h}}=s$. Now define $\widetilde{d}$ as the composition of $\alpha'$ with $I \hookrightarrow M \rtimes \mathfrak{g} \rightarrow M$. It is clear that this is the derivation we are looking for.
\end{proof}
We state the following consequences. 

\begin{corollary}
The map $\hat{\tr}$ yields a well-defined map $\tr : H^1(\Lh,M)^{\Lg / \Lh} \to H^2(\Lg/\Lh,M^{\Lh})$.
\end{corollary}
\begin{proof}
Since we already proved that $\hat{\tr}$ is a group homomorphism, it is sufficient to prove that $\hat{\tr}(d)=0$ if $d : \mathfrak{h} \rightarrow M$ is an inner derivation. Since $M$ is a $\mathfrak{g}$-module, it is immediate that in this case we can extend $d$ to an inner derivation of $\mathfrak{g}$. Together with Lemma \ref{extend_derivation}, this finishes the proof.
\end{proof}

\begin{corollary}
The sequence 
\[\xymatrix{0 \ar[r] & H^1(\mathfrak{g}/\mathfrak{h},M^{\mathfrak{h}}) \ar[r]^-{\infl} & H^1(\mathfrak{g},M) \ar[r]^-{\res} & H^1(\mathfrak{h},M)^{\mathfrak{g}/\mathfrak{h}} \ar[r]^-{\tr} & H^2(\mathfrak{g}/\mathfrak{h},M^{\mathfrak{h}})}\]
 is exact.
\end{corollary}
\begin{proof}
This follows readily from Lemma \ref{extend_derivation} once one observes that, on derivation level, the map $\res$ takes a derivation of $\Lg$ to its restriction to $\Lh$. 
\end{proof}

\begin{proposition}\label{tr_good}
The map $\tr : H^1(\mathfrak{h},M)^{\mathfrak{g}/\mathfrak{h}} \rightarrow H^2(\mathfrak{g}/\mathfrak{h},M^{\mathfrak{h}})$ is a linear map of vector spaces.
\end{proposition}
\begin{proof}
We already showed in Corollary \ref{tr_group_hom} that $\tr$ preserves addition. It remains to show that $\tr$ preserves scalar multiplication with elements of the ring $k$.
Take $\lambda \in k$ and observe that the map $\lambda : M \rightarrow M$, $m \mapsto \lambda \cdot m$ is a $\mathfrak{g}$-module morphism. Hence we can take the restriction $\lambda: M^{\mathfrak{h}} \rightarrow M^{\mathfrak{h}}$. The induced maps $\lambda_*: H^1(\mathfrak{h},M)^{\mathfrak{g}/\mathfrak{h}} \rightarrow H^1(\mathfrak{h},M)^{\mathfrak{g}/\mathfrak{h}}$ and $\lambda_*: H^2(\mathfrak{g}/\mathfrak{h},M^{\mathfrak{h}}) \rightarrow H^2(\mathfrak{g}/\mathfrak{h},M^{\mathfrak{h}})$ give precisely the scalar multiplication on the cohomology spaces. By Lemma \ref{tr_natural}, $\tr$ is natural with respect to the modules, so $\tr$ is a linear map of vector spaces.
\end{proof}

\begin{lemma}\label{exactness_lemma}
Let $\underline{e}: \xymatrix@1{0 \ar[r] & M \ar[r]^i & \mathfrak{e} \ar[r]^p & \mathfrak{g} \ar[r]& 0}$ be an extension of $\mathfrak{g}$ by $M$. 
\begin{itemize}
\item If $\underline{e}$ is partially split and there exists a partial splitting $s$ of $\underline{e}$ with $d^s_x$ an inner derivation for all $x \in \mathfrak{e}$, then $\infl(\omega(\underline{e},s))=[\underline{e}]$. 
\item Conversely, if there exists an $[\underline{e'}] \in H^2(\mathfrak{g}/\mathfrak{h},M^{\mathfrak{h}})$ such that $[\underline{e}]=\infl[\underline{e'}]$, then there exists a partial splitting $s$ of $\underline{e}$ with $d^s_x$ an inner derivation for all $x \in \mathfrak{e}$, such that  $[\underline{e'}]=\omega(\underline{e},s)$.
\end{itemize}
\end{lemma}
\begin{proof}
The first part of the lemma immediately follows from the fact that the inflation map is the composition $H^1(\mathfrak{g}/\mathfrak{h},M^{\mathfrak{h}}) \rightarrow H^1(\mathfrak{g},M^{\mathfrak{h}}) \rightarrow H^1(\mathfrak{g},M)$. On extension level, these maps are represented by respectively a pull-back and a push-out construction. Together with the remarks about the construction of $\omega$, this yields the first property.

Conversely, if $[\underline{e}]=\infl[\underline{e}']$, there is a commuting diagram
\[\xymatrix{\underline{e}' : 0 \ar[r] & M^{\mathfrak{h}} \ar@{=}[d]\ar[r] & \mathfrak{l} \ar[r] & \mathfrak{g}/\mathfrak{h} \ar[r] & 0\\
{\phantom{\underline{e} ': } 0} \ar[r]& M^{\mathfrak{h}} \ar@{^(->}[d]\ar[r] & \mathfrak{p} \ar[u]^{\beta} \ar[d]^-{\gamma} \ar[r] & \mathfrak{g} \ar[u] \ar@{=}[d] \ar[r] & 0\\
\underline{e}\phantom{'} :  0 \ar[r] & M \ar[r] & \mathfrak{e} \ar[r] & \mathfrak{g} \ar[r] & 0,}\]
where the upper right hand square is a pull-back, and the lower left hand square is a push-out construction. 
Using the pull-back property, we can find a partial splitting $s': \mathfrak{h} \rightarrow \mathfrak{p}$ of the second row, such that $\beta \circ s'=0$. 
Define $s = \gamma \circ s'$, a partial splitting of $\underline{e}$.
The short exact sequence $\xymatrix@1{0 \ar[r] & \mathfrak{h} \ar[r] & \mathfrak{g} \ar[r] & \mathfrak{g}/\mathfrak{h} \ar[r] & 0}$ induces an exact sequence 
\[\xymatrix{0 \ar[r] & \mathfrak{h} \ar[r]^{s'} & \mathfrak{p} \ar[r] & \mathfrak{l} \ar[r] & 0}\] by Proposition \ref{pb}, so it follows that $s'(\mathfrak{h})$ is an ideal in $\mathfrak{p}$. As a consequence, one can check that the middle row of the diagram induces an exact sequence
\[\underline{e}'' : \xymatrix{0 \ar[r] & M^{\mathfrak{h}} \ar[r] & \mathfrak{p}/s'(\mathfrak{h}) \ar[r] & \mathfrak{g}/\mathfrak{h}  \ar[r] & 0.}\] 
%meaning 2
It also means that $\gamma$ has image in $I_{\mathfrak{e}}(s(\Lh))$, so the map $I_{\mathfrak{e}}(s(\Lh)) \rightarrow \mathfrak{g}$ is surjective. 
By Lemma \ref{surjective}, $d^s_x$ is an inner derivation for all $x \in \mathfrak{e}$. We show that $\omega(\underline{e},s)=[\underline{e}']$. 
Since $\beta \circ s'=0$, we can take the quotient map $\overline{\beta}: \mathfrak{p}/s'(\Lh) \rightarrow \mathfrak{l}$, making the diagram
\[\xymatrix{\underline{e}'':  0 \ar[r] & M^{\mathfrak{h}} \ar@{=}[d]\ar[r] & \mathfrak{p}/s'(\mathfrak{h}) \ar[d]^{\overline{\beta}}\ar[r] & \mathfrak{g}/\mathfrak{h} \ar@{=}[d] \ar[r] & 0\\
\underline{e}':  0 \ar[r]& M^{\mathfrak{h}} \ar[r] & \mathfrak{l} \ar[r] & \mathfrak{g}/\mathfrak{h}  \ar[r] & 0}\]
commute. 
%first part
Moreover, since $\gamma$ has image in $I_{\mathfrak{e}}(s(\Lh))$ and $s = \gamma \circ s'$, we can take the quotient map
$\overline{\gamma} : \mathfrak{p}/s'(\mathfrak{h}) \rightarrow I_{\mathfrak{e}}(s(\Lh))/s(\mathfrak{h})$. 
This gives an equivalence of extensions.
\[\xymatrix{{\phantom{\omega s  s s}}\underline{e}'':  0 \ar[r] & M^{\mathfrak{h}} \ar@{=}[d]\ar[r] & \mathfrak{p}/s'(\mathfrak{h}) \ar[d]^-{\overline{\gamma}}\ar[r] & \mathfrak{g}/\mathfrak{h} \ar@{=}[d] \ar[r] & 0\\
\omega(\underline{e},s):  0 \ar[r]& M^{\mathfrak{h}} \ar[r] & I_{\mathfrak{e}}(s(\Lh))/s(\mathfrak{h}) \ar[r] & \mathfrak{g}/\mathfrak{h}  \ar[r] & 0.}\]
This shows that both $\underline{e}'$ and $\omega(\underline{e},s)$ are equivalent to the same extension, hence they are themselves equivalent. 
\end{proof}
We can prove the following corollaries of Lemma \ref{exactness_lemma}.
\begin{corollary}
The sequence 
\[\xymatrix{H^1(\mathfrak{h},M)^{\mathfrak{g}/\mathfrak{h}} \ar[r]^-{\tr} & H^2(\mathfrak{g}/\mathfrak{h},M^{\mathfrak{h}})\ar[r]^-{\infl} & H^2(\mathfrak{g},M)}\]
is exact.
\end{corollary}
\begin{proof}
Since $\tr[d]=\omega(\underline{e}_0,s_d)$, it follows directly from the previous lemma that $\infl \circ \tr [d]=[\underline{e}_0]=0$. Conversely, if $\infl[\underline{e}]=[\underline{e}_0]$, there exists an appropriate partial splitting $s$ of $\underline{e}_0$ such that $[\underline{e}]=\omega(\underline{e}_0,s)$. If $d$ is the splitting associated to $s$, it is clear that $[\underline{e}]=\tr[d]$.
\end{proof}

\begin{corollary}\label{cor_image_infl}
Take $[\underline{e}] \in H^2(\mathfrak{g},M)_1$. Now $[\underline{e}] \in \Image \infl$ iff there exists a partial splitting $s: \mathfrak{h} \rightarrow M$ such that for all $x \in \mathfrak{e}$, $d^s_x$ is an inner derivation. 
\end{corollary}

\begin{corollary}
The image of $\infl$ is contained in $H^2(\mathfrak{g},M)_1$.
\end{corollary}

\section{The map $\rho$}\label{second_map}
We construct a map $H^2(\mathfrak{g},M)_1 \rightarrow H^1(\mathfrak{g}/\mathfrak{h}, H^1(\mathfrak{h},M))$. The class of an extension
\[\underline{e}: \xymatrix{0 \ar[r] & M \ar[r]^{i} & \mathfrak{e} \ar[r]^{p} & \mathfrak{g} \ar[r] & 0}\]  belongs to $H^2(\mathfrak{g},M)_1$ if and only $\underline{e}$ is partially split, i.e.~there exists a partial splitting $s_0 : \mathfrak{h} \rightarrow \mathfrak{e}$. Take such an extension $\underline{e}$ and fix a partial splitting $s_0$. Define a map $\widetilde{\rho}(\underline{e}) : \mathfrak{e} \rightarrow \mbox{Der}(\mathfrak{h},M)$ by putting $\widetilde{\rho}(\underline{e})(x)=d^{s_0}_x$, which was defined as \[d^{s_0}_x(z)=[x,s_0(z)]-s_0[p(x),z]\] for all $z \in \mathfrak{h}$. 
For simplicity of notation, we write $d_x=d^{s_0}_x$ if no other splitting than $s_0$ is considered. We immediately get the following result.
\begin{lemma}\label{rho_pre_image}
The class $[d_x] \in H^1(\mathfrak{h},M)$ depends only on $\pi \circ p (x)$.
\end{lemma}
\begin{proof}
Since $M$ is abelian, it is easy to see that $d_{m+x}-d_x=-\delta(m)$, where $\delta(m)$ is the inner derivation mapping $z \in \mathfrak{h}$ to $z \cdot m$. This means that $[d_x]$ depends only on $p(x)$. Now for every element $z \in \mathfrak{h}$, we can choose the pre-image $s_0(z) \in \mathfrak{e}$. Since $s_0$ is a homomorphism of Lie algebras, it is clear that $d_{s_0(z)}=0$. This shows that if $p(x')=p(x) + z$ for some $z \in \Lh$, $[d_x]=[d_{x'}]$, so the class of the derivation $d_x$ only depends on $\pi \circ p(x)$.
\end{proof}

Observe that we have also proven that, if $[d_x]=[d]$, then there exists some $y \in \Le$ with $\pi \circ p(y)=\pi \circ p (x)$ such that $d = d_y$. 

As a consequence, we can define a map 
\[\overline{\rho}(\underline{e}) : \mathfrak{g}/\mathfrak{h} \rightarrow H^1(\mathfrak{h},M),\] mapping $x \in \mathfrak{g}/\mathfrak{h}$ to the class $[d_y] \in H^1(\mathfrak{h},M)$ with $\pi \circ p (y) = x$.

\begin{lemma}
The map $\overline{\rho}(\underline{e})$ is a derivation with respect to the standard $\Lg / \Lh$-action on $H^1(\Lh,M)$.
\end{lemma}
\begin{proof}
It suffices to show that $\widetilde{\rho}(\underline{e}) : \mathfrak{e} \rightarrow \mbox{Der}(\mathfrak{h},M)$ is a derivation, where the $\mathfrak{e}$-action on $\mbox{Der}(\mathfrak{h},M)$ is defined via $p$. 
Take $y_1, \, y_2 \in \mathfrak{e}$. Once one realizes that the $\mathfrak{g}$-module structure on $M$ is the one induced by the Lie bracket in $\mathfrak{e}$, it is an easy calculation to show that $d_{[y_1,y_2]} = \, ^{p(y_1)}d_{y_2} - \,  ^{p(y_2)}d_{y_1}$. 
\end{proof}

This means that we can take the equivalence class $[\overline{\rho}(\underline{e})]$ of $\overline{\rho}(\underline{e})$, and consider it as an element of $H^1(\mathfrak{g}/\mathfrak{h},H^1(\mathfrak{h},M))$.
Now we will show that $[\overline{\rho}(\underline{e})]$ is independent of the chosen partial splitting $s_0$. Even stronger, by adjusting the choice of the partial splitting, we can get any derivation in the equivalence class of $\overline{\rho}(\underline{e})$. Recall that adding a derivation to a splitting yields another splitting, and that the difference of two splittings is always a derivation. 

\begin{lemma} \label{change_section}
If $s_0$ and $s_1$ are partial sections of $\underline{e}$ with $s_1=s_0 + d$ and associated maps  $\widetilde{\rho}_0(\underline{e})$, $\widetilde{\rho}_1(\underline{e}): \mathfrak{e} \rightarrow \mbox{Der}(\mathfrak{h},M)$, 
then $\widetilde{\rho}_1(\underline{e})(x)-\widetilde{\rho}_0(\underline{e})(x)=
{}^{p(x)}d$ for all $x \in \mathfrak{e}$.
\end{lemma}
\begin{proof}
Fix $x \in \mathfrak{e}$ and write $d_x^0=\widetilde{\rho}_0(\underline{e})(x)$ and $d_x^1=\widetilde{\rho}_1(\underline{e})(x)$. It is easy to see that, for all $z \in \mathfrak{h}$, $(d_x^1 - d_x^0)(z)=[x,d(z)]-d[p(x),z]=(^{p(x)}d)(z)$.
\end{proof}

We have proven that the difference $\widetilde{\rho}_1(\underline{e})-\widetilde{\rho}_0(\underline{e})$ equals the inner derivation $\delta(d)$ in $\mbox{Der}(\mathfrak{e},\mbox{Der}(\mathfrak{h},M))$, where the $\mathfrak{e}$-action on $\mbox{Der}(\mathfrak{h},M)$ is again induced by $p$. It follows immediately that $\overline{\rho}_1(\underline{e})-\overline{\rho}_0(\underline{e})=\delta[d] \in \mbox{Der}(\mathfrak{g}/\mathfrak{h},H^1(\mathfrak{h},M))$. This means that the equivalence class $[\overline{\rho}(\underline{e})]$ is independent of the chosen section. Hence, we get a well-defined map $\hat{\rho}$, mapping an extension $\underline{e}$ to the class $[\overline{\rho}(\underline{e})] \in H^1(\mathfrak{g}/\mathfrak{h}, H^1(\mathfrak{h},M))$. It remains to show that this gives rise to a well-defined homomorphism of cohomology spaces, that will make the sequence exact. 

\begin{lemma}\label{well_natural}
Suppose that $M_1$ and $M_2$ are two $\mathfrak{g}$-modules, $\alpha : M_1 \rightarrow M_2$ is a $\mathfrak{g}$-module homomorphism, and $\underline{e}$ and $\underline{e}'$ are partially split extensions of $\mathfrak{g}$ by $M_1$ respectively $M_2$ that fit in a commutative diagram 
\[\xymatrix{\underline{e}: 0 \ar[r] & M_1 \ar[d]^{\alpha}\ar[r]^{i} & \mathfrak{e} \ar[d]^{\beta} \ar[r]^{p} & \mathfrak{g} \ar[r]& 0\\ \underline{e}': 0 \ar[r] & M_2 \ar[r]^{i'} & \mathfrak{e}' \ar[r]^{p'} & \ar@{=}[u] \mathfrak{g} \ar[r]& 0,}\]
where $\beta$ is a Lie algebra morphism.
Then $\alpha_* (\hat{\rho}(\underline{e}))=\hat{\rho}(\underline{e}')$, where $\alpha_*$ is the induced map $H^1(\mathfrak{g}/\mathfrak{h},H^1(\mathfrak{h},M_1)) \rightarrow H^1(\mathfrak{g}/\mathfrak{h},H^1(\mathfrak{h},M_2))$. 
\end{lemma}
Observe that $\underline{e}'$ makes the diagram commutative if and only if the left hand square is a push-out construction, and equivalently, $[\underline{e}']=H^2(\mathbb{1}, \alpha)[\underline{e}]$.
This follows from Proposition \ref{exact_is_poc}.

\begin{proof}
Choose a partial splitting $s : \mathfrak{h} \rightarrow \mathfrak{e}$ for $\underline{e}$ and take the partial splitting $s'=\beta \circ s$ for $\underline{e}'$.
Take $x \in \mathfrak{g}/\mathfrak{h}$ and choose a pre-image $y \in \mathfrak{e}$ under $\pi \circ p$. This means that $\hat{\rho}(\underline{e})(x)=[d_y]$, and $\alpha_*(\hat{\rho}(\underline{e}))(x)=[\alpha \circ d_y]$, with $i'\circ \alpha \circ d_y (z)=i ' \circ \alpha ([y,s(z)] - s[py,z]) = %\beta \circ i  ([y,s(z)] - s[py,z]) 
 [\beta y, s' (z)]-s'[p(y),z]$ for $z \in \mathfrak{h}$.
On the other hand, we can choose $\beta y \in \mathfrak{e}'$ as a pre-image of $x$ under $\pi \circ p'$, and $\hat{\rho}(\underline{e}')(x)=[d_{\beta y}]$ with $i'\circ d_{\beta y}(z)=[\beta y,s'(z)]-s'[p'(\beta y),z]$ for $z \in \mathfrak{h}$. Since $p' \circ \beta=p$, this finishes the proof.
\end{proof}

We state the following immediate consequences.
\begin{corollary}
The map $\hat{\rho}$ yields a well-defined map $\rho : H^2(\mathfrak{g},M)_1  \rightarrow  H^1(\mathfrak{g}/\mathfrak{h},H^1(\mathfrak{h},M))$.
\end{corollary}

\begin{proof}
If two extensions $\underline{e}$ and $\underline{e}'$ are equivalent, we can use Lemma \ref{well_natural} with $\alpha = \mathbb{1}_M$ to see that $\hat{\rho}(\underline{e})=\hat{\rho}(\underline{e}')$.
\end{proof}

\begin{corollary}\label{rho_natural}
The map $\rho$ is natural with respect to the modules.
\end{corollary}

\begin{proof}
We have to show that for every $\Lg$-module map $\alpha : M_1 \rightarrow M_2$, the diagram
\[\xymatrix{H^2(\Lg,M_1)_1 \ar[r]^-{\rho} \ar[d]^{\alpha_*} &  H^1(\Lg/ \Lh, H^1(\Lh, M_1)) \ar[d]^{\alpha_*} \\ H^2(\Lg, M_2) \ar[r]^-{\rho} & H^1(\Lg/ \Lh, H^1(\Lh, M_2))}\]
commutes. Take $\underline{e}$ and $\underline{e}'$ as in Lemma \ref{well_natural}. Then it is clear that $[\underline{e}']=\alpha_*[\underline{e}]$. The lemma shows that $\hat{\rho}(\underline{e}')=\alpha_*( \hat{\rho}(\underline{e}))$, which means exactly that the diagram commutes.
\end{proof}

\begin{proposition}\label{rho_good}
The map $\rho : H^2(\mathfrak{g},M)_1 \rightarrow H^1(\mathfrak{g}/\mathfrak{h},H^1(\mathfrak{h},M))$ is a linear map of vector spaces.
\end{proposition}
The fact that $\rho$ preserves sums is now easily proven using Lemma \ref{they_preserve} and Lemma \ref{homomorphism_of_group_objects}. Linearity follows from Lemma \ref{rho_natural} in the same way as in the proof of Proposition \ref{tr_good}.

\begin{lemma}
The sequence 
\[\xymatrix{\ldots \ar[r] & H^1(\mathfrak{h},M)^{\mathfrak{g}/\mathfrak{h}} \ar[r]^{\tr} & H^2(\mathfrak{g}/\mathfrak{h},M^{\mathfrak{h}}) \ar[r]^{\infl} & H^2(\mathfrak{g},M)_1 \ar[r]^{\rho \ \ \ \ \ \ } & H^1(\mathfrak{g}/\mathfrak{h},H^1(\mathfrak{h},M))}\]
is exact.
\end{lemma}
\begin{proof}
If the extension class $[\underline{e}]$ is in the image of $\infl$, we choose the splitting $s$ from Corollary \ref{cor_image_infl}. It is then immediate that $\overline{\rho}(\underline{e})$ with respect to $s$ is trivial, so $\rho [\underline{e}]=0$. 
Conversely, take $\underline{e}$ such that $\rho[\underline{e}]$=0. By the remarks preceding Lemma \ref{change_section}, we can assume that $\overline{\rho}(\underline{e}) = 0$. In this case, by the remarks following Lemma \ref{rho_pre_image}, there is a splitting $s$ of $\underline{e}$ such that $d^s_x$ is an inner derivation for all $x \in \mathfrak{e}$. Using Corollary \ref{cor_image_infl}, we see that $[\underline{e}] \in \Image (\infl)$.
\end{proof}

\section{Interpretation of the third cohomology group}\label{third_cohom}
To construct the map $H^1(\mathfrak{g}/\mathfrak{h}, H^1(\mathfrak{h},M)) \rightarrow H^3(\Lg / \Lh, M^{\Lh})$ explicitly, we need to have an interpretation of the third cohomology group for Lie algebras. We give here a short overview; a more detailed version can for example be found in \cite[Section 1]{wage06-1}. 

We first need to define crossed modules of Lie algebras, as introduced in the appendix of \cite{kl82-1}. A \emph{crossed module} of Lie algebras consists of a Lie algebra homomorphism $\delta : \mathfrak{m} \rightarrow \mathfrak{n}$, together with a Lie algebra action of $\Ln$ on $\Lm$ (i.e.~a Lie algebra map $\mathfrak{n} \rightarrow \mbox{Der}(\mathfrak{m})$), where the action on an element will be noted as $^x y$, such that 
\begin{itemize}
	\item $\delta(^n m)=[n, \delta(m)]$ for all $n \in \mathfrak{n}$ and $m\in \mathfrak{m}$;
	\item $^{\delta(m)} m'=[m,m']$ for all $m, \ m' \in \mathfrak{m}$.
\end{itemize}
It is an immediate consequence of the definition that $\Image \delta$ is an ideal of $\mathfrak{n}$, $M = \Ker \delta$ is central in $\mathfrak{m}$ and therefore abelian, and that the action of $\mathfrak{n}$ on $\mathfrak{m}$ induces an action of the cokernel $\mathfrak{q}=\mathfrak{n}/\Image \delta$ on $M$. 
A homomorphism of crossed modules is a pair $(\lambda, \mu) : (\Lm \overset{\delta}{\rightarrow} \Ln) \rightarrow (\Lm' \overset{\delta'}{\rightarrow} \Ln')$, where $\lambda : \Lm \rightarrow \Lm'$ and $\mu : \Ln \rightarrow \Ln'$ are Lie algebra maps such that the diagram \[\xymatrix{\Lm \ar[d]^{\lambda} \ar[r]^{\delta} & \Ln \ar[d]^{\mu}\\ \Lm' \ar[r]^{\delta'} & \Ln'}\] commutes and $\lambda$ preserves the action, i.e.~ $\lambda({}^x y) = {}^{\mu(x)} \lambda(y)$.

A \emph{crossed extension} of $\Lq$ by $M$ is an exact sequence of Lie algebras
\[\underline{e} : \xymatrix{0 \ar[r] & M \ar[r]^{i} & \Lm \ar[r]^{\delta} & \Ln \ar[r]^{\pi} & \Lq \ar[r] & 0, }\] where $M$ is an abelian Lie algebra and $\delta : \Lm \rightarrow \Ln$ is a crossed module. The crossed module induces a Lie algebra action of $\Lq$ on $M$. 
We define an equivalence relation on the crossed extensions, as follows. If $\underline{e}$ and $\underline{e}'$ are two crossed extensions with kernel $M$ and quotient $\Lq$, and there is a crossed module map $(\lambda, \mu)$ making the diagram
\begin{equation}\label{equivalence} \xymatrix{\underline{e} : 0 \ar[r] & M \ar[r]^{i} \ar@{=}[d] & \Lm \ar[r]^{\delta} \ar[d]^{\lambda} & \Ln \ar[d]^{\mu} \ar[r]^{\pi} & \Lq \ar@{=}[d] \ar[r] & 0 \\\underline{e}': 0 \ar[r] & M \ar[r]^{i'} & \Lm' \ar[r]^{\delta'} & \Ln' \ar[r]^{\pi'} & \Lq \ar[r] & 0 }\end{equation} commute, then the crossed extensions are equivalent. However, since $\lambda$ and $\mu$ don't need to be isomorphisms, we have to take the equivalence relation generated by this notion. This means that two crossed extensions $\underline{e}$ and $\underline{e}'$ are equivalent if and only if there exists a series of crossed extensions $\underline{e}=\underline{e}_0$, $\underline{e}_1$, \ldots, $\underline{e}_n = \underline{e}'$ where the adjacent extensions are connected by crossed module maps, creating commutative diagrams as (\ref{equivalence}). 

In \cite[Section 1]{wage06-1}, Wagemann uses an explicit construction of the cocycle to prove that the equivalence classes of crossed extensions of $\Lq$ by $M$ are indeed in bijection with the third homology class $H^3(\Lq,M)$.
The first (non-explicit) proofs of this correspondence can be found in \cite{gers64-1} and \cite{gers66-1}. 

Just as in Proposition \ref{exact_is_pb} and \ref{exact_is_poc}, we can use the pull-back and the push-out construction to describe the induced maps on cohomology. Take a crossed extension $\underline{e}$ as above. We can decompose $\underline{e}$ into a Lie algebra extension with trivial induced action
\[\underline{e}_1: \xymatrix{0 \ar[r] & M \ar[r]^{i} & \Lm \ar[r] & \Ker \pi \ar[r] & 0,}\]
and an extension of Lie algebras with non-abelian kernel
\[\underline{e}_2 : \xymatrix{0 \ar[r] & \Ker \pi \ar[r] & \Ln \ar[r] & \Lq \ar[r] & 0.}\] The induced maps $\alpha^* : H^3(\Lq, M) \rightarrow H^3(\Lq',M)$ and  $\beta_* : H^3(\Lq, M) \rightarrow H^3(\Lq, M')$ on cohomology can be constructed like this: we alter $\underline{e}_2$ respectively $\underline{e}_1$ like in Section \ref{pb_po}, and then paste it together with the other part again. One can state similar results as in Section \ref{pb_po}, but we leave it to the reader to make it explicit.

\section{Construction of $\lambda$}\label{lambda}

We are now ready to construct the map $\lambda: H^1(\Lg / \Lh,H^1(\Lh,M)) \to H^3(\Lg /\Lh,M^{\Lh})$. For every Lie algebra $\mathfrak{g}$ and ideal $\mathfrak{h}$, and every $\mathfrak{g}$-module $M$, we can make a crossed extension
\[ \xymatrix{0 \ar[r] & M^{\mathfrak{h}} \ar[r] & M \rtimes \mathfrak{h} \ar[r]^-{\gamma} & \mbox{Der}(\mathfrak{h},M) \rtimes \mathfrak{g} \ar[r]^-{\Pi} & H^1(\mathfrak{h},M) \rtimes \mathfrak{g}/\mathfrak{h} \ar[r] & 0,}\] where the middle terms form a crossed module by putting $^{(d,x)}(m,z)=(x \cdot m + d(z), [x,z])$ and $\gamma(m,z)=(-\delta(m),z)$ with $\delta(m)(z)=z \cdot m$. The map $\Pi$ comes from the natural projection maps. It is straight-forward to check that this forms indeed a crossed extension. 

Given a derivation $D : \mathfrak{g}/\mathfrak{h} \rightarrow H^1(\mathfrak{h},M)$, we consider the splitting $s_D : \mathfrak{g}/\mathfrak{h} \rightarrow H^1(\mathfrak{h},M) \rtimes \mathfrak{g}/\mathfrak{h}$, mapping $x$ to $(D(x),x)$. Taking the pull-back $\mathfrak{p}$ of $s_D$ and $\Pi$, it is easy to check that we get an exact sequence
\[\underline{e}_D : \xymatrix{0 \ar[r] & M^{\mathfrak{h}} \ar[r] & M \rtimes \mathfrak{h} \ar[r]^-{\gamma} & \mathfrak{p} \ar[r] & \mathfrak{g}/\mathfrak{h} \ar[r] & 0,}\] that is a crossed extension of $\mathfrak{g}/\mathfrak{h}$ by $M$, and the induced action on $M^{\Lh}$ coincides with the $\Lg / \Lh$-action on $M^{\Lh}$ given by the module action of $\Lg$ on $M$. Observe that $\mathfrak{p}$ is nothing but the pre-image of $\Image (s_D)$ under $\Pi$, since $s_D$ is injective. It consists of all couples $(d,x)$ with $D(x+\mathfrak{h})=[d]$. We define the map \[\hat{\lambda}: \mbox{Der}(\mathfrak{g}/\mathfrak{h},H^1(\mathfrak{h},M)) \rightarrow H^3(\mathfrak{g}/\mathfrak{h},M^{\mathfrak{h}})\] by $\hat{\lambda}(D)=[\underline{e}_D]$. We have to show that this yields a well-defined homomorphism on cohomology level, that is natural in the modules and makes the seven-term sequence exact.

%well-defined
\begin{lemma}
The map $\hat{\lambda}$ yields a well-defined map $\lambda:H^1(\mathfrak{g}/\mathfrak{h},H^1(\mathfrak{h},M)) \rightarrow H^3(\mathfrak{g}/\mathfrak{h},M^{\mathfrak{h}})$.
\end{lemma}
\begin{proof}
Take $D'=D + \delta[\widetilde{d}]$ and take the pull-back $\mathfrak{p}$ of $\Pi$ and $s_D$, and the pull-back $\mathfrak{p}'$ of $\Pi$ and $s_{D'}$.
Then $\mathfrak{p}$ consists of all couples $(d,x)$ with $D(x + \mathfrak{h})=[d]$, and $\mathfrak{p}'$ consists of all $(d,x)$ with $D(x+\mathfrak{h})+[^x \widetilde{d}]=[d]$. We have to show that the two associated crossed extensions 
are equivalent. Define a map $\alpha : \mathfrak{p} \rightarrow \mathfrak{p}'$ by putting $\alpha(d,x)=(d + {}^x \widetilde{d},x)$. It is straight-forward to show that this is a Lie algebra homomorphism with image in $\mathfrak{p}'$. Now take $\beta : M \rtimes \Lh \rightarrow M \rtimes \Lh$ with $\beta(m,x)=(m-\widetilde{d}(x),x)$. It is easy to see that $\beta$ is a Lie algebra morphism such that $(\beta,\alpha)$ is a homomorphism of crossed modules, and that we get a commutative diagram 
\[\xymatrix{\underline{e}_D :  0 \ar[r] & M^{\mathfrak{h}} \ar[r] \ar@{=}[d]& M \rtimes \mathfrak{h} \ar[d]^{\beta}\ar[r]^-{\gamma} & \mathfrak{p} \ar[d]^{\alpha}\ar[r] & \mathfrak{g}/\mathfrak{h}\ar@{=}[d] \ar[r] & 0\\
\underline{e}_{D'}:  0 \ar[r] & M^{\mathfrak{h}} \ar[r] & M \rtimes \mathfrak{h} \ar[r]^-{\gamma'} & \mathfrak{p}' \ar[r] & \mathfrak{g}/\mathfrak{h} \ar[r] & 0.}\] This shows that $\underline{e}_D$ and $\underline{e}_{D'}$ are equivalent.
\end{proof}

%natural
\begin{lemma}\label{lambda_natural}
The map $\lambda : H^1(\mathfrak{g}/\mathfrak{h},H^1(\mathfrak{h},-)) \rightarrow H^3(\mathfrak{g}/\mathfrak{h},-^{\mathfrak{h}})$ is a natural transformation of functors. 
\end{lemma}
\begin{proof}
Take a $\mathfrak{g}$-module morphism $\alpha : M_1 \rightarrow M_2$.
We show that the diagram
\[\xymatrix{H^1(\mathfrak{g}/\mathfrak{h},H^1(\mathfrak{h},M_1))\ar[d]^{\alpha_*} \ar[r]^-{\lambda} & H^3(\mathfrak{g}/\mathfrak{h},M_1^{\mathfrak{h}}) \ar[d]^{\alpha_*}\\
H^1(\mathfrak{g}/\mathfrak{h},H^1(\mathfrak{h},M_2)) \ar[r]^-{\lambda} & H^3(\mathfrak{g}/\mathfrak{h},M_2^{\mathfrak{h}})}\]
is commutative. 

Take $[D] \in H^1(\mathfrak{g}/\mathfrak{h},H^1(\mathfrak{h},M_1))$. Observe that $\alpha_*[D]=[D']$, where $D'=H^1(\mathbb{1},\alpha)(D)$, so if $D(x)=[d]$, then $D'(x)=[\alpha \circ d]$. If $\mathfrak{p}'$ is the pull-back of $\Pi$ and $s_{D'}$ as before, then $\lambda \circ \alpha_*[D]$ is represented by 
\[\xymatrix{0 \ar[r] & M_2^{\mathfrak{h}} \ar[r] & M_2 \rtimes \mathfrak{h} \ar[r] & \mathfrak{p}' \ar[r] & \mathfrak{g}/\mathfrak{h} \ar[r] & 0.}\]
On the other hand, one can check that $\alpha_* \circ \lambda[D]$ can be represented by the bottom row in the diagram
\[\xymatrix{ 0 \ar[r] & M_1^{\mathfrak{h}} \ar[r] \ar[d]^{\alpha} & M_1 \rtimes \mathfrak{h} \ar[d]\ar[r] & \mathfrak{p} \ar@{=}[d] \ar[r] & \mathfrak{g}/\mathfrak{h}\ar@{=}[d] \ar[r] & 0\\
 0 \ar[r] & M_2^{\mathfrak{h}} \ar[r] & \mathfrak{e} \ar[r] & \mathfrak{p} \ar[r] & \mathfrak{g}/\mathfrak{h} \ar[r] & 0,}\]
where the left hand square is a push-out construction as described in Section \ref{pb_po}. We show that the two crossed extensions are equivalent. Take maps  $\alpha \rtimes \mathbb{1} : M_1 \rtimes \mathfrak{h} \rightarrow M_2 \rtimes \mathfrak{h}$ and $M_2^{\mathfrak{h}} \hookrightarrow M_2 \rtimes \mathfrak{h}$, and use the universal property of the push-out construction to obtain a map $\rho : \mathfrak{e} \rightarrow M_2 \rtimes \mathfrak{h}$. We can also define a map $\tau:\mathfrak{p} \rightarrow \mathfrak{p}'$ mapping $(d,x)$ to $(\alpha \circ d,x)$. 
Now one can check that the diagram 
\[\xymatrix{ 0 \ar[r] & M_2^{\mathfrak{h}} \ar@{=}[d] \ar[r] & \mathfrak{e} \ar[r] \ar[d]^{\rho} & \mathfrak{p} \ar[d]^{\tau} \ar[r] & \mathfrak{g}/\mathfrak{h} \ar@{=}[d] \ar[r] & 0\\
 0 \ar[r] & M_2^{\mathfrak{h}} \ar[r] & M_2 \rtimes \mathfrak{h} \ar[r] & \mathfrak{p}'  \ar[r] & \mathfrak{g}/\mathfrak{h} \ar[r] & 0,}\]
is commutative, using uniqueness in the universal property of the push-out construction.
\end{proof}

%homomorphism

\begin{proposition}
The map $\lambda : H^1(\mathfrak{g}/\mathfrak{h},H^1(\mathfrak{h},-)) \rightarrow H^3(\mathfrak{g}/\mathfrak{h},-^{\mathfrak{h}})$ is a homomorphism of vector spaces.
\end{proposition}
\begin{proof}
The map $\lambda$ is a group homomorphism by Lemma \ref{they_preserve}, Lemma \ref{homomorphism_of_group_objects} and Lemma \ref{lambda_natural}. The fact that $\lambda$ preserves multiplication follows directly from naturality, in the same way as in the previous sections.
\end{proof}

To prove exactness, we need the following lemma. 
\begin{lemma}\label{Huebschmann}
A crossed extension 
\[\underline{e}:\xymatrix{0 \ar[r] & M \ar[r] & \mathfrak{m} \ar[r]^{\delta} & \mathfrak{g} \ar[r]^{\pi} & \mathfrak{q} \ar[r] & 0}\]
 is equivalent to the zero extension if and only if there exists  a short exact sequence \break $\underline{e}':\xymatrix@1{0 \ar[r]& \mathfrak{m} \ar[r]^{i} & \mathfrak{e} \ar[r]& \mathfrak{q} \ar[r]& 0}$ of Lie algebras and a Lie algebra homomorphism $h : \mathfrak{e} \rightarrow \mathfrak{g}$ such that the diagram
\[\xymatrix{ \underline{e}': {\phantom{0}}  & 0 \ar[r]& \mathfrak{m} \ar@{=}[d]\ar[r]^{i}& \mathfrak{e} \ar[d]^{h}\ar[r]& \mathfrak{q} \ar@{=}[d] \ar[r]& 0\\
\underline{e} {\phantom{'}}:  0 \ar[r]& M \ar[r]& \mathfrak{m} \ar[r]^{\delta} & \mathfrak{g} \ar[r]^{\pi} & \mathfrak{q} \ar[r] & 0}\]
commutes and $(\mathbb{1}_{\mathfrak{m}},h)$ is a homomorphism of crossed modules.
\end{lemma}
The proof of the lemma will be postponed to Section \ref{huebschmann_section}. Now we can prove exactness of the sequence.

%exactness
\begin{lemma}
The sequence 
\[\xymatrix{H^2(\mathfrak{g},M)_1\ar[r]^-{\rho}& H^1(\mathfrak{g}/\mathfrak{h},H^1(\mathfrak{h},M)) \ar[r]^-{\lambda} & H^3(\mathfrak{g}/\mathfrak{h},M^{\mathfrak{h}})}\] is exact.
\end{lemma}
\begin{proof}
Take a partially split extension
\[\underline{e} : \xymatrix{0 \ar[r] & M \ar[r]^-{i} & \mathfrak{e} \ar[r]^-{p} & \mathfrak{g} \ar[r] & 0.}\] Since the sequence is partially split, there is an injective map $j : M \rtimes \mathfrak{h} \rightarrow \mathfrak{e}$ such that 
$p \, \circ j = p_0$ and $j \circ i_0=i$, where $p_0$ is the surjection $M \rtimes \mathfrak{h} \rightarrow \mathfrak{h}$ and $i_0$ is the embedding $M \hookrightarrow M \rtimes \mathfrak{h}$. It is not difficult to show that this yields an exact sequence 
\[\underline{e}' : \xymatrix{0 \ar[r] & M \rtimes \mathfrak{h} \ar[r]^-{j} & \mathfrak{e} \ar[r]^-{\pi \circ p} & \mathfrak{g}/\mathfrak{h} \ar[r] & 0.}\] 
Take the partial splitting $s(z)=j(0,z)$ of $\underline{e}$, and the associated derivation $D$ representing $\rho[\underline{e}]$. Now $\lambda \circ \rho [\underline{e}]$ can be represented by 
\[\xymatrix{0 \ar[r] & M^{\mathfrak{h}} \ar[r] & M \rtimes \mathfrak{h} \ar[r]^-{\gamma} & \mathfrak{p} \ar[r] & \mathfrak{g}/\mathfrak{h} \ar[r] & 0,}\]
where $\mathfrak{p} \subset \mbox{Der}(\mathfrak{h},M) \rtimes \mathfrak{g}$ consists of all pairs $(d,x)$ for which $[d]=D(x+\mathfrak{h})$, i.e.~$[d]=[d^s_{\widetilde{x}}]$ for some $\widetilde{x}$ with $p(\widetilde{x})=x$. 
It is now straight-forward that there is a map $h : \mathfrak{e} \rightarrow \mathfrak{p}$; $y \mapsto (d^s_y,p(y))$, that gives a commutative diagram
\[\xymatrix{ & 0 \ar[r]& M \rtimes \mathfrak{h} \ar@{=}[d]\ar[r]& \mathfrak{e} \ar[d]^{h}\ar[r]& \mathfrak{g}/\mathfrak{h} \ar@{=}[d] \ar[r]& 0\\
0 \ar[r]& M^{\mathfrak{h}} \ar[r]& M \rtimes \mathfrak{h} \ar[r]^{\gamma} & \mathfrak{p} \ar[r] & \mathfrak{g}/\mathfrak{h} \ar[r] & 0,}\] where the left hand square is a morphism of crossed modules. This can easily be proven using the relation $j(m,z)=i(m)+s(z)$. Lemma \ref{Huebschmann} shows that $\lambda \circ \rho [\underline{e}]=0$.

Conversely, take $[D]\in H^1(\mathfrak{g}/\mathfrak{h},H^1(\mathfrak{h},M))$ with $\lambda[D]=0$. This means that there exists a short exact sequence 
$\xymatrix@1{0 \ar[r]& M \rtimes \mathfrak{h} \ar[r]^-{i} & \mathfrak{e} \ar[r]^-{p} & \mathfrak{g}/\mathfrak{h} \ar[r]& 0}$ and a Lie algebra homomorphism $h : \mathfrak{e} \rightarrow \mathfrak{p}$ such that the diagram
\[\xymatrix{ & 0 \ar[r]& M \rtimes \mathfrak{h} \ar@{=}[d]\ar[r]^-{i}& \mathfrak{e} \ar[d]^{h}\ar[r]^-{p}& \mathfrak{g}/\mathfrak{h} \ar@{=}[d] \ar[r]& 0\\
0 \ar[r]& M^{\mathfrak{h}} \ar[r]& M \rtimes \mathfrak{h} \ar[r]^-{\gamma} & \mathfrak{p} \ar[r] & \mathfrak{g}/\mathfrak{h} \ar[r] & 0}\]
is commutative, where $\mathfrak{p}$ is the usual pull-back and the left hand square is a morphism of crossed modules. We take the composition $\mathfrak{p} \hookrightarrow \mbox{Der}(\mathfrak{h},M) \rtimes \mathfrak{g} \rightarrow \mathfrak{g}$, and compose this with $h$ to get a map $p': \mathfrak{e} \rightarrow \mathfrak{g}$. It is not difficult to prove that this yields an exact sequence 
\[\underline{e}: \xymatrix{0 \ar[r] & M \ar[r]^{i \circ i_0} & \mathfrak{e} \ar[r]^{p'} & \mathfrak{g} \ar[r] & 0,}\] where $i_0$ is the inclusion $M \hookrightarrow M \rtimes \mathfrak{h}$. The sequence is partially split with partial splitting $s(z)=i(0,z)$. Take the associated derivation $D'$ representing $\rho[\underline{e}]$. We claim that $D'=D$. For $x \in \mathfrak{g}$, we choose a pre-image $y \in \mathfrak{e}$ under $p'$. Take the derivation $d$ such that $h(y)=(d,x) \in \mbox{Der}(\mathfrak{h},M) \rtimes \mathfrak{g}$. Since $h(y)$ belongs to $\mathfrak{p}$, we know that $[d]=D(x + \mathfrak{h})$. 
Now 
\[i \circ i_0 \circ d^s_y(z)=[y,s(z)]-s[p'(y),z]=[y,i(0,z)]-i(0,[x,z])=i(^{h(y)}(0,z))-i(0,[x,z]),\]
where the last equality holds because $h$ is a morphism of crossed modules. It follows that $i_0 \circ d^s_y(z)=(d(z),[x,z])-(0,[x,z])=(d(z),0)$, so $d^s_y=d$. This finishes the proof, since $D'(x + \mathfrak{h})=[d^s_y]=[d]=D(x + \mathfrak{h})$ for all $x \in \Lg$, so $[D]=[D']=\rho[\underline{e}]$.
\end{proof}

\section{Naturality of $\tr$, $\rho$ and $\lambda$}

We have shown that the maps $tr$, $\rho$ and $\lambda$ we have constructed are natural in the modules. A question that arises is whether the maps are also natural in the Lie algebra extensions. 

Take two Lie algebra extensions $\underline{\epsilon}$ and $\underline{\epsilon}'$ and Lie algebra maps $\alpha$ and $\overline{\alpha}$ making the following diagram commute:
\[\xymatrix{ \underline{\epsilon}':  0 \ar[r] &\mathfrak{h}' \ar[d]^{\alpha} \ar[r] & \mathfrak{g}' \ar[d]^{\alpha}\ar[r]^-{\pi'} & \mathfrak{g}'/\mathfrak{h}' \ar[d]^{\overline{\alpha}}\ar[r] & 0\\
\underline{\epsilon} {\phantom{'}}: 0 \ar[r] &\mathfrak{h} \ar[r] & \mathfrak{g} \ar[r]^-{\pi} & \mathfrak{g}/\mathfrak{h} \ar[r] & 0.}\] Take a $\mathfrak{g}$-module $M$ and observe that it inherits a $\mathfrak{g}'$-module structure via $\alpha$.

\subsection{The map $\tr$}
It is easy to show that $\alpha$ induces a map $\alpha^*: H^1(\mathfrak{h},M)^{\mathfrak{g}/\mathfrak{h}} \rightarrow H^1(\mathfrak{h}',M)^{\mathfrak{g}'/\mathfrak{h}'}$, since $^x (d \circ \alpha)={}^{\alpha(x)}d \circ \alpha $ for $x \in \mathfrak{g}'$ and $d : \mathfrak{h} \rightarrow M$ a derivation. Now we show that the diagram
\[\xymatrix{H^1(\mathfrak{h},M)^{\mathfrak{g}/\mathfrak{h}} \ar[d]^{\alpha^*} \ar[r]^-{\tr}& H^2(\mathfrak{g}/\mathfrak{h},M^{\mathfrak{h}}) \ar[d]^{\alpha^*}\\
H^1(\mathfrak{h}',M)^{\mathfrak{g}'/\mathfrak{h}'} \ar[r]^-{\tr} & H^2(\mathfrak{g}'/\mathfrak{h}',M^{\mathfrak{h}'})}\]
is commutative. 
Take $[d] \in H^1(\mathfrak{h},M)^{\mathfrak{g}/\mathfrak{h}}$ and put $d'=d \circ \alpha$. It is clear that $\tr \circ \, \alpha^*[d]$ can be represented by the extension
\begin{equation}\label{tr_circ_alpha}\underline{e}: \xymatrix{0 \ar[r] & M^{\mathfrak{h}'} \ar[r] & I'/s'(\mathfrak{h}') \ar[r] & \mathfrak{g}'/\mathfrak{h}' \ar[r] & 0,}\end{equation} where $s' : \mathfrak{h}' \rightarrow M \rtimes \mathfrak{g}'$ is the partial splitting associated with $d'$, and $I'\subset M \rtimes \mathfrak{g}'$ is the largest Lie subalgebra in which $s'(\mathfrak{h}')$ is an ideal. 
On the other hand, we can represent $\alpha^* \circ \tr[d]$ by the bottom row in the diagram
\begin{equation}\label{p12}\xymatrix{0 \ar[r] & M^{\mathfrak{h}} \ar[r] & I / s(\mathfrak{h}) \ar[r] & \mathfrak{g}/\mathfrak{h} \ar[r] & 0\\
0 \ar[r] & M^{\mathfrak{h}} \ar@{=}[u] \ar@{^(->}[d]^{j} \ar[r] & \mathfrak{p} \ar[u] \ar[d]\ar[r] & \mathfrak{g}'/\mathfrak{h}' \ar[u]^{\overline{\alpha}} \ar@{=}[d] \ar[r] & 0\\
0 \ar[r] & M^{\mathfrak{h}'} \ar[r] & \mathfrak{e}\ar[r] & \mathfrak{g}'/\mathfrak{h}' \ar[r] & 0,}\end{equation}
where $s$ is the partial splitting associated with $d$ and $I \subset M \rtimes \Lg$ is the largest Lie subalgebra in which $s(\Lh)$ is an ideal. The upper right hand square is a pull-back and the lower left hand square is a push-out construction. Observe that we have indeed an inclusion $j:M^{\mathfrak{h}} \hookrightarrow M^{\mathfrak{h}'}$, since the action of $\mathfrak{h}'$ is induced by the $\mathfrak{h}$-action via $\alpha$. We want to find some extension $\underline{e}'$ of $\mathfrak{g}'/\mathfrak{h}'$ by $M^{\mathfrak{h}}$, that is equivalent to the second row in the previous diagram and such that $[\underline{e}]=j_*[\underline{e}']$.
First we need the following lemma. 
\begin{lemma}\label{char_I}
An element $(m,x) \in M \rtimes \mathfrak{g}$ belongs to $I_{M \rtimes \mathfrak{g}}(s_d(\mathfrak{h}))$ iff $^{x}d=\delta(m)$. 
\end{lemma}
We can consider the map $\mathbb{1} \rtimes \alpha : M \rtimes \mathfrak{g}' \rightarrow M \rtimes \mathfrak{g}$. 
The lemma shows that if $(m,\alpha(x))\in I \subset M\rtimes \mathfrak{g}$, then $(m,x) \in I' \subset M \rtimes \mathfrak{g}'$, so we can define $\mathfrak{s} := (\mathbb{1} \rtimes \alpha)^{-1}(I)$ in $I'$. This means that $s'(\mathfrak{h}')$ is an ideal in $\mathfrak{s}$. It is not difficult to show that there is now an exact sequence 
\[\xymatrix{0 \ar[r] & M^{\mathfrak{h}}\ar[r] & \mathfrak{s} \ar[r] & \mathfrak{g}' \ar[r] & 0}\] that yields an exact sequence 
\[\underline{e}': \xymatrix{0 \ar[r] & M^{\mathfrak{h}}\ar[r] & \mathfrak{s}/s'(\mathfrak{h}') \ar[r]^{p'} & \mathfrak{g}'/\mathfrak{h}' \ar[r] & 0.}\] 
Since we have a commutative diagram of exact sequences
\[\xymatrix{\underline{e}' :  0 \ar[r] & M^{\mathfrak{h}} \ar@{^(->}[d] \ar[r] & \mathfrak{s}/s'(\mathfrak{h}') \ar@{^(->}[d] \ar[r] & \mathfrak{g}'/\mathfrak{h}' \ar@{=}[d] \ar[r] & 0\\
\underline{e} {\phantom{'}}:  0 \ar[r] & M^{\mathfrak{h}'} \ar[r] & I'/s'(\mathfrak{h}') \ar[r] & \mathfrak{g}'/\mathfrak{h}' \ar[r] & 0
,}\]
it is immediate that $\underline{e}$ is equivalent to the push-out construction of $\underline{e}'$ under the inclusion $M^{\Lh} \hookrightarrow M^{\Lh'}$.
Furthermore, we can take the quotient map $\rho_1 : \mathfrak{s} / s'(\mathfrak{h}') \rightarrow I / s(\mathfrak{h})$ of $\mathbb{1} \rtimes \alpha$. If we take $\rho_2=p' : \mathfrak{s} / s'(\mathfrak{h}') \rightarrow \Lg' / \Lh'$, the pull-back property gives a map $\rho : \mathfrak{s}/s'(\mathfrak{h}') \rightarrow \mathfrak{p}$, where $\Lp$ is the Lie algebra appearing in diagram (\ref{p12}). 
It is left to the reader to check that this makes the 
 diagram
\[\xymatrix{0 \ar[r] & M^{\mathfrak{h}} \ar@{=}[d] \ar[r] & \mathfrak{s}/s'(\mathfrak{h}) \ar[d]^{\rho} \ar[r] & \mathfrak{g}'/\mathfrak{h}' \ar@{=}[d] \ar[r] & 0\\
0 \ar[r] & M^{\mathfrak{h}} \ar[r] & \mathfrak{p} \ar[r] & \mathfrak{g}'/\mathfrak{h}' \ar[r] & 0}\] commutative, so the two extensions are equivalent.

\subsection{The map $\rho$}
Observe that $\alpha$ induces a map $\alpha^* : H^2(\mathfrak{g},M)_1 \rightarrow H^2(\mathfrak{g}',M)_1$.
We prove that the diagram 
\[\xymatrix{H^2(\mathfrak{g},M)_1 \ar[d]^{\alpha^*} \ar[r]^-{\rho}& H^1(\mathfrak{g}/\mathfrak{h},H^1(\mathfrak{h}, M)) \ar[d]^{\alpha^*}\\
H^2(\mathfrak{g}',M)_1 \ar[r]^-{\rho}& H^1(\mathfrak{g}'/\mathfrak{h}',H^1(\mathfrak{h}', M)) }\] commutes.
Take $[\underline{e}] \in H^2(\mathfrak{g},M)_1$ and set $[\underline{e}']=\alpha^*[\underline{e}]$. We can represent $\underline{e}'$ by the bottom row in the diagram
\[\xymatrix{\underline{e}{\phantom{'}} :  0 \ar[r] & M \ar[r] & \mathfrak{e} \ar[r]^{p} & \mathfrak{g} \ar[r] & 0\\
\underline{e}' :  0 \ar[r] & M \ar@{=}[u]\ar[r] & \mathfrak{l} \ar[u]^{\beta} \ar[r]^{p'} & \mathfrak{g}' \ar[u]^{\alpha} \ar[r] & 0,}\]
where the right hand square is a pull-back. If we fix a partial splitting $s : \mathfrak{h} \rightarrow \mathfrak{e}$ for $\underline{e}$, we can use the property of the pull-back to construct a partial splitting $s'$ for $\underline{e}'$ with $\beta \circ s'=s \circ \alpha$. 
Take the derivation $D'$ representing $\rho[\underline{e}']$, associated to the splitting $s'$. It is easy to check that $d^{s'}_x=d^s_{\beta x} \circ \alpha$ for all $x \in \mathfrak{l}$. It follows that $\alpha^*[d^s_{\beta x}]=[d^{s'}_x]=D'(\pi ' p' (x))$ for all $x \in \mathfrak{l}$, while 
$[d^s_{\beta x}]=D(\pi p \beta (x))=(D \circ \overline{\alpha})(\pi ' p'(x))$. Since $\pi'$ and $p'$ are surjective, this means that $D'=H^1(\mathbb{1},\alpha^*) \circ D \circ \overline{\alpha}$, and this shows precisely that $[D']=\alpha^* [D]$. 

\subsection{The map $\lambda$}
We show that the diagram
\[\xymatrix{H^1(\mathfrak{g}/\mathfrak{h},H^1(\mathfrak{h}, M)) \ar[d]^{\alpha^*} \ar[r]^-{\lambda}& H^3(\mathfrak{g}/\mathfrak{h},M^{\mathfrak{h}}) \ar[d]^{\alpha^*}\\
H^1(\mathfrak{g}'/\mathfrak{h}',H^1(\mathfrak{h}', M)) \ar[r]^-{\lambda}& H^3(\mathfrak{g}'/\mathfrak{h}',M^{\mathfrak{h}'}) }\] commutes. 
Take $[D] \in H^1(\mathfrak{g}/\mathfrak{h},H^1(\mathfrak{h}, M))$. 
If we define $D'=H^1(\mathbb{1},\alpha^*) \circ D \circ \overline{\alpha}$, it is clear that $\alpha^*[D]=[D']$. The projections $\Der (\Lh, M) \rtimes \Lg \rightarrow H^1(\Lh,M) \rtimes \Lg/\Lh $ and $\Der (\Lh', M) \rtimes \Lg' \rightarrow H^1(\Lh',M) \rtimes \Lg'/\Lh' $ are called $\Pi$ and $\Pi'$ respectively. Let  $\mathfrak{p}$ be the pull-back of $s_{D}$ and $\Pi$, and $\mathfrak{p}'$ the pull-back of $s_{D'}$ and $\Pi'$. 
Then $\lambda \circ \alpha^*[D]$ can be represented by
\[\underline{e} : \xymatrix{0 \ar[r] & M^{\mathfrak{h}'} \ar[r] & M \rtimes \mathfrak{h}' \ar[r] & \mathfrak{p}' \ar[r] & \mathfrak{g}'/\mathfrak{h}' \ar[r] & 0.} \]
On the other hand, we can represent $\alpha^* \circ \lambda[D]$ by the lower crossed extension in the diagram
\begin{equation}\label{dia}\xymatrix{ {\phantom{\underline{e}':}} 0 \ar[r] & M^{\mathfrak{h}} \ar@{=}[d]\ar[r] & M \rtimes \mathfrak{h} \ar@{=}[d] \ar[r] & \mathfrak{p} \ar[r] & \mathfrak{g}/\mathfrak{h} \ar[r] & 0 \\
{\phantom{\underline{e}':}} 0 \ar[r] & M^{\mathfrak{h}} \ar@{^(->}^{j}[d] \ar[r] & M \rtimes \mathfrak{h} \ar[d] \ar[r] & \mathfrak{e} \ar[u] \ar@{=}[d]\ar[r] & \mathfrak{g}'/\mathfrak{h}' \ar[u]^{\overline{\alpha}}\ar@{=}[d] \ar[r] & 0\\
\underline{e}':  0 \ar[r] & M^{\mathfrak{h}'} \ar[r] & \mathfrak{l} \ar[r] & \mathfrak{e} \ar[r] & \mathfrak{g}'/\mathfrak{h}' \ar[r] & 0,}\end{equation}
where the upper right hand square is a pull-back, and the lower left hand square is a push-out construction. Our goal is to find a crossed extension $\underline{e}''$ that is equivalent to the second row of the preceding diagram, such that $[\underline{e}]=j_*[\underline{e}'']$. In this case, we find $[\underline{e}] = [\underline{e}']$ as desired.

Observe that there is an exact sequence 
\[\xymatrix{0 \ar[r] & M^{\mathfrak{h}} \ar[r] & M \rtimes \mathfrak{h}' \ar[r]^-{\gamma''} & \mbox{Der}(\mathfrak{h},M) \rtimes \mathfrak{g}' \ar[r]^-{\Pi''} & H^1(\mathfrak{h},M) \rtimes \mathfrak{g}'/\mathfrak{h}' \ar[r] & 0,}\] where $\gamma '' (m,z)=(-\delta(m),z)$,
 $\mbox{Der}(\mathfrak{h},M)$ is a $\mathfrak{g}'$-module via $\alpha$, and $\Pi''$ is the projection map.
Define an action of $\mbox{Der}(\mathfrak{h},M) \rtimes \mathfrak{g}'$ on $M \rtimes \mathfrak{h}'$ by $^{(d,x)}(m,z)=(x \cdot m + d(\alpha(z)),[x,z])$. It is easy to show that this makes $\gamma '' : M \rtimes \mathfrak{h}' \rightarrow \mbox{Der}(\mathfrak{h},M) \rtimes \mathfrak{g}'$ into a crossed module. Take the pull-back $\mathfrak{p}''$ of $s_{D \circ \overline{\alpha}}$ and $\Pi''$. We claim that the sequence
\[\underline{e}'': \xymatrix{0 \ar[r] & M^{\mathfrak{h}} \ar[r] & M \rtimes \mathfrak{h}' \ar[r]^-{\gamma''} & \mathfrak{p}'' \ar[r]^-{\pi''} & \mathfrak{g}'/\mathfrak{h}' \ar[r] & 0}\] is the crossed extension we look for. 
To show that $\underline{e}''$ is equivalent to the second row in diagram (\ref{dia}), we have to construct a map $\beta: \mathfrak{p}'' \rightarrow \mathfrak{e}$. Using the pull-back property, take $\beta_1 = \mathbb{1} \rtimes \alpha : \mathfrak{p}'' \rightarrow \mathfrak{p}$ and $\beta_2 =\pi'':\mathfrak{p}'' \rightarrow \mathfrak{g}'/\mathfrak{h}'$ to construct $\beta$. One can check that we get a commuting diagram
\[\xymatrix{\underline{e}'':  0 \ar[r] & M^{\mathfrak{h}} \ar@{=}[d] \ar[r] & M \rtimes \mathfrak{h}' \ar[d]^{\mathbb{1} \rtimes \alpha} \ar[r] & \mathfrak{p}'' \ar[r]^-{\pi''} \ar[d]^{\beta} & \mathfrak{g}'/\mathfrak{h}' \ar@{=}[d]\ar[r] & 0\\
 {\phantom{\underline{e}'':}}0 \ar[r] & M^{\mathfrak{h}} \ar[r] & M \rtimes \mathfrak{h} \ar[r] & \mathfrak{e} \ar[r] & \mathfrak{g}'/\mathfrak{h}' \ar[r] & 0.}\] Furthermore, $(\mathbb{1} \rtimes \alpha,\beta)$ forms a morphism of crossed modules, so we indeed have an equivalence of crossed extensions. 
Now we show that $[\underline{e}]=j_*[\underline{e}'']$. It suffices to show that there exist maps $\nu$ and $\eta$ such that the diagram 
\[\xymatrix{\underline{e}'' :  0 \ar[r] & M^{\mathfrak{h}} \ar@{^(->}[d] \ar[r] & M \rtimes \mathfrak{h}' \ar[d] \ar[r] & \mathfrak{p}'' \ar[r] \ar@{=}[d] & \mathfrak{g}'/\mathfrak{h}' \ar@{=}[d]\ar[r] & 0\\
{\phantom{\underline{e}'':}} 0 \ar[r] & M^{\mathfrak{h}'} \ar@{=}[d] \ar[r] & \Ln \ar[d]^{\nu} \ar[r] & \mathfrak{p}'' \ar[r] \ar[d]^{\eta} & \mathfrak{g}'/\mathfrak{h}' \ar@{=}[d]\ar[r] & 0\\
\underline{e}{\phantom{''}}:  0 \ar[r] & M^{\mathfrak{h}'} \ar[r] & M \rtimes \mathfrak{h}' \ar[r] & \mathfrak{p}' \ar[r] & \mathfrak{g}'/\mathfrak{h}' \ar[r] & 0,}\]
commutes, where the upper left hand square is a push-out construction and $(\nu, \eta)$ is a morphism of crossed extensions.
Using the universal property of the push-out construction, we define $\nu : \Ln \rightarrow M \rtimes \mathfrak{h}'$ as the map that is the identical map on $M \rtimes \mathfrak{h}'$ and the inclusion $M^{\mathfrak{h}'} \hookrightarrow M \rtimes \Lh'$ on $M^{\mathfrak{h}'}$. Furthermore, one can check that the map $\mbox{Der}(\alpha,\mathbb{1}) \rtimes \mathbb{1}$ restricts to a map $\eta: \mathfrak{p}'' \rightarrow \mathfrak{p}'$, such that $(\nu, \eta)$ is a morphism of crossed modules. 
This finishes the proof.
 
\section{Cocycle description}\label{cocycle}
In this section, we use the notations as introduced for the construction of the maps. 
Fix a linear map $\alpha : \mathfrak{g}/\mathfrak{h} \rightarrow \mathfrak{g}$ such that $\pi \circ \alpha = \mbox{id}_{\mathfrak{g}/\mathfrak{h}}$, and set $f_{\alpha}(x \wedge y)=[\alpha(x),\alpha(y)]-\alpha[x,y] \in \mathfrak{h}$ for $x, \, y \in \mathfrak{g}/\mathfrak{h}$.

\subsection{The map $\tr$, first description}
Take a derivation $d : \Lh \rightarrow M$ with $[d] \in H^1(\Lh,M)^{\Lg/\Lh}$ and denote the associated partial splitting of $\underline{e}_0$ by $s_d$. Choose a section $s: \Lg \rightarrow I_{M \rtimes \Lg}(s_d(\Lh))$ of the extension
\[\xymatrix{0 \ar[r] & M^{\Lh} \ar[r]^-i & I_{M \rtimes \Lg}(s_d(\Lh)) \ar[r] & \Lg \ar[r] & 0,}\] and take the associated factor set $f_s : \Lg \wedge \Lg \rightarrow M^{\Lh}$, with $i \circ f_s(x \wedge y)=[s(x),s(y)]-s[x,y]$. Then the map $\overline{s} : \Lg / \Lh \wedge \Lg / \Lh \rightarrow I_{M \rtimes \Lh}(s_d(\Lh))$, defined by $\overline{s}(x \wedge y)=s( \alpha (x) \wedge \alpha(y))+s_d(\Lh)$, is a section of the extension $\underline{e}_d$ that represents $\tr[d]$. A straight-forward calculation shows that 
\[[s \circ \alpha(x), s \circ \alpha(y)]-s \circ \alpha[x,y]=s \circ f_{\alpha}(x \wedge y) + i \circ f_s(\alpha(x), \alpha(y)).\]
It follows that the cocycle $F : \Lg/\Lh \wedge \Lg / \Lh \rightarrow M^{\Lh}$ associated to $\underline{e}_d$ can be written as 
\[F(x \wedge y)=f_s(\alpha(x) \wedge \alpha(y))) + \overline{i}^{-1}(s \circ f_{\alpha}(x \wedge y) + s_d(\Lh)),\] where $\overline{i}$ denotes the injective map $M^{\Lh} \rightarrow I_{M \rtimes \Lg}(s_d(\Lh))/s_d(\Lh)$. If $s$ extends $s_d$, the expression becomes even simpler, since in this case we can write
\[F(x \wedge y)=f_s(\alpha(x),\alpha(y)).\]

\subsection{The map $\tr$, second description}
Take $[d] \in H^1(\mathfrak{h},M)^{\mathfrak{g}/\mathfrak{h}}$. 
Fix a map $\eta : \mathfrak{g}/\mathfrak{h} \rightarrow M$ such that for all $x \in \mathfrak{g}/\mathfrak{h}$, $^{\alpha(x)} d = \delta(\eta(x))$, where $\delta$ maps an element of $M$ to the associated inner derivation as before. By Lemma \ref{char_I}, we can take a section $\overline{s}(x)=(\eta(x),\alpha(x)) + s_d(\mathfrak{h})$ of the sequence 
\[\underline{e}':\xymatrix{0 \ar[r] & M^{\Lh} \ar[r]^-{\overline{i}} & I_{M \rtimes \mathfrak{g}}(s(\mathfrak{h}))/s_d(\mathfrak{h}) \ar[r] & \mathfrak{g}/\mathfrak{h} \ar[r] & 0.}\]
Using the correspondence of extensions and 2-cocycles, we get a cocycle $F : \mathfrak{g}/\mathfrak{h} \wedge \mathfrak{g}/\mathfrak{h} \rightarrow M^{\Lh}$ corresponding to $\underline{e}'$, defined by 
\[\overline{i}F(x \wedge y)=[\overline{s}(x),\overline{s}(y)]-\overline{s}[x,y].\] An easy calculation shows that 
\[\overline{i}F(x \wedge y)=\big(\alpha(x) \cdot \eta(y) - \alpha(y) \cdot \eta(x) - \eta[x,y], f_{\alpha}(x \wedge y)\big) + s_d(\mathfrak{h}).\] 
Since $((d \circ f_{\alpha})(x \wedge y),f_{\alpha}(x\wedge y))\in s_d(\mathfrak{h})$, it follows that 
\[F(x \wedge y)=\alpha(x) \cdot \eta(y) - \alpha(y) \cdot \eta(x) - \eta[x,y] - (d \circ f_{\alpha})(x \wedge y)\]
and $\tr[d]=[F]$. Observe that the first three terms are a formal coboundary expression of $\eta$ via $\alpha$.

\subsection{The map $\rho$}
Take $[f] \in H^2(\mathfrak{g},M)_1$, so $f$ is a bilinear alternating map $f : \mathfrak{g} \wedge \mathfrak{g} \rightarrow M$ and $f_{|\mathfrak{h} \wedge \mathfrak{h}} = \delta^1 \gamma$, where $\delta^1$ is the first coboundary map and $\gamma : \mathfrak{h} \rightarrow M$ is some linear map.
The cocycle corresponds to an extension
\[\xymatrix{0 \ar[r] & M \ar[r]^-{i} & M \times_f \mathfrak{g} \ar[r] & \mathfrak{g} \ar[r] & 0,}\]
where the Lie algebra in the middle has underlying vector space $M \times \mathfrak{g}$ and the Lie bracket is defined by $[(m,x),(m',x')]=(x \cdot m' - x' \cdot m + f(x \wedge x'), [x,x'])$. We choose a partial splitting $s_0 : \mathfrak{h} \rightarrow M$, defined by $z \mapsto (\gamma(z),z)$, which is a Lie algebra morphism since $f_{|\mathfrak{h} \wedge \mathfrak{h}} = \delta^1(\gamma)$. 
For $x \in \mathfrak{g}/\mathfrak{h}$, choose $(0,\alpha(x)) \in M \times_f \mathfrak{g}$ as pre-image under $\pi \circ p$. Observe that  \[d_{(0,\alpha(x))}(z)=i^{-1}\Big([(0,\alpha(x)),s_0(z)]-s_0[p(0,\alpha(x)),z]\Big)=f(\alpha(x) \wedge z) + \alpha(x) \cdot \gamma(z) - \gamma[\alpha(x),z].\]
It follows that we can represent $\rho[f]$ by the derivation $D : \mathfrak{g}/\mathfrak{h} \rightarrow H^1(\mathfrak{h},M)$, mapping $x$ to the class $[D_x]$, with $D_x(z)=f(\alpha(x) \wedge z)+ \alpha(x) \cdot \gamma(z) - \gamma[\alpha(x),z]$ for all $z \in \mathfrak{h}$. 

\subsection{The map $\lambda$}
Fix a map $s : H^1(\mathfrak{h},M) \rightarrow \mbox{Der}(\mathfrak{h},M)$ that is a section of the natural projection, and take a derivation $D$ with $[D] \in H^1(\Lg/\Lh,H^1(\Lh,M))$. It is easy to see that the map that takes $x$ to $(sD(x ), \alpha(x))$ is a section of the last map in the exact sequence 
\[\underline{e}_D:\xymatrix{0 \ar[r] & M^{\mathfrak{h}} \ar[r] & M \rtimes \mathfrak{h} \ar[r] & \mathfrak{p} \ar[r] & \mathfrak{g}/\mathfrak{h} \ar[r] & 0,}\] so we want to look at 
\[f(x , y)=[(sD(x),\alpha(x)),(sD(y),\alpha(y))]-(sD[x,y],\alpha[x,y]).\]
Observe that for all $x, \, y \in \Lg/\Lh$, the expression ${}^{\alpha(x)}sD(y) - ^{\alpha(y)}sD(x)-sD[x,y]$ represents an inner derivation, so we can fix a map $F : \mathfrak{g}/\mathfrak{h} \times  \mathfrak{g}/\mathfrak{h} \rightarrow M$ such that
\[\delta F(x , y)={}^{\alpha(x)}sD(y) - ^{\alpha(y)}sD(x)-sD[x,y],\]
where $\delta$ maps an element of $M$ to the associated inner derivation as before.
The map $F$ can be chosen to be linear and alternating, and it measures the defect of $s \circ D$ being a ``derivation'' with respect to the ``action'' of $\Lg / \Lh$ on $\Der (\Lh,M)$ defined through $\alpha$ (although this does not really define an action). 
Now it is straight-forward to see that $f(x , y)=(\delta F(x , y),f_{\alpha}(x , y))$. We can lift this towards a bilinear, alternating map $f' : \mathfrak{g}/\mathfrak{h} \times  \mathfrak{g}/\mathfrak{h} \rightarrow M \rtimes \mathfrak{h}$ by taking $f'(x , y)=(-F(x , y),f_{\alpha}(x, y))$. 
It is easy to show that the associated cocycle equals 
\[c(x \wedge y \wedge z)=-c'(x \wedge y \wedge z)+sD(x)(f_{\alpha}(y \wedge z)) - sD(y)(f_{\alpha}(x \wedge z)) + sD(z)(f_{\alpha}(x \wedge y)),\]
with 
\[\begin{array}{rcl}c'(x \wedge y \wedge z)&=&\alpha(x) \cdot F(y , z)-\alpha(y) \cdot F(x , z)+ \alpha(z) \cdot F(x , y) \\ & & - F([x,y] , z)+F([x,z] , y)-F([y,z] , x),\end{array}\]
the formal coboundary expression of $F$.

\section{Example: The Heisenberg Lie Algebras}\label{heisenberg_example}
In this section, we will not distinguish in notation between cocycles and their equivalence classes. 

Let $\mathfrak{g}_m$ be a Heisenberg Lie algebra over a field $k$: as a vector space it has basis \break $\{x, x_1, \cdots, x_m, y_1, y_2, \cdots, y_m\}$ and the only non-trivial Lie brackets in the generators are $[x_i,y_i]=x$ for $1 \leq i \leq m$. It is clear that the Lie subalgebra $\mathfrak{h}$ generated by $x$ is a one-dimensional ideal of $\mathfrak{g}_m$, while the quotient $\mathfrak{g}_m/\mathfrak{h}$ is a $2m$-dimensional abelian Lie algebra. It follows immediately that $\dim_k H^1(\mathfrak{g}_m/\mathfrak{h},k)=2m$, $\dim_k H^2(\mathfrak{g}_m/\mathfrak{h},k)=\binom{2m}{2}$ and $\dim_k H^3(\mathfrak{g}_m/\mathfrak{h},k) = \binom{2m}{3}$. The generators are respectively all maps $f_a : \mathfrak{g}_m / \Lh \rightarrow k$, $f_{a \wedge b} : \Lambda^2 (\mathfrak{g}_m / \Lh) \rightarrow k$, and $f_{a \wedge b \wedge c} : \Lambda^3(\mathfrak{g}_m /\Lh) \rightarrow k$, for all $a, \ b, \ c \in \{x_1, \cdots, x_m, y_1, \cdots, y_m\}$, each map sending the basis element in the subscript to 1, and the other basis elements of $\Lambda^i (\mathfrak{g}_m / \Lh)$ to zero. 

One can easily check (see also \cite{sant83-1}) that $H^1(\mathfrak{g}_m,k)$ has dimension $2m$, and consists of all $k$-homomorphisms $\mathfrak{g}_m \rightarrow k$ sending $x$ to zero. Moreover, $H^2(\mathfrak{g}_m, k)_1$ equals $H^2(\mathfrak{g}_m,k)$ since $\mathfrak{h}$ is a free Lie algebra on one element. We can choose basis elements $f_{a \wedge b}$ for all $a, b \in \{x_1, \cdots, x_m, y_1, \cdots, y_m\}$ with $a \wedge b \neq x_m \wedge y_m$. (Observe that $\sum_{i=1}^m f_{x_i \wedge y_i}$ is the coboundary of $f_x$ and therefore yields zero in the cohomology group.)
 
The action of $\mathfrak{g}_m / \mathfrak{h}$ on $H^1(\mathfrak{h},k)$ is trivial since $\mathfrak{h}$ is central in $\mathfrak{g}_m$, so $H^1(\mathfrak{h},k)^{\mathfrak{g}_m / \mathfrak{h}} \cong k$. Now we are ready to examine the maps in the exact sequence
\[\xymatrix{0 \ar[r] & H^1(\mathfrak{g}_m/\mathfrak{h},k) \ar[r]^-{\infl} & H^1(\mathfrak{g}_m,k) \ar[r]^-{\res} & H^1(\mathfrak{h},k)^{\mathfrak{g}_m / \mathfrak{h}} \ar[r]^{\tr} & H^2(\mathfrak{g}_m/\mathfrak{h},k) \ar[r]^-{\infl} & \break
H^2(\mathfrak{g}_m,k)_1} \]
\[\xymatrix{ \ar[r]^-{\rho}& H^1(\mathfrak{g}_m/\mathfrak{h},H^1(\mathfrak{h},k)) \ar[r]^-{\lambda} & H^3(\mathfrak{g}_m/\mathfrak{h},k)}.\]  
It is clear that the first morphism maps basis elements on basis elements and therefore yields the identical map after identification of both groups with $\bigoplus_{i=1}^{2m} k$. The second map is clearly trivial, whereas the second inflation map is identical on the basis elements $f_{a \wedge b}$ for $a \wedge b \neq x_m \wedge y_m$, and maps $f_{x_m \wedge y_m}$ onto $-\sum_{i=1}^{m-1} f_{x_i \wedge y_i}$. 

To find the other maps, we work with the cocycle descriptions, using the notations of Section \ref{cocycle}. If $\alpha : \mathfrak{g}_m / \mathfrak{h} \rightarrow \mathfrak{g}_m$ is the linear map that takes the basis elements of $ \mathfrak{g}_m / \mathfrak{h}$ to the corresponding basis elements in $\mathfrak{g}_m$, it is easy to see that $f_{\alpha}(x_i \wedge y_i)=1 \in k$ and $f_{\alpha}$ is zero on the other basis elements. We can choose $\eta \equiv 0$, so it follows that $\tr(f_x)=-\sum_{i=1}^m f_{x_i \wedge y_i}$. Since $\mathfrak{h}$ is generated by $x$, and $f_{a \wedge b}(c \wedge x)=0$ for all basis elements $c \in \mathfrak{g}_m$, it follows that $\rho \equiv 0$. To describe the last map, observe that $H^1(\mathfrak{h},k)=\mbox{Der}(\mathfrak{h},k)$, so we can take $s=\mbox{id}$ such that $F \equiv 0$ and $c' \equiv 0$. Denote by $D_{x_i}$ (resp.~$D_{y_i}$) the derivation mapping $x_i$ (resp.~$y_i$) to $f_x$ and mapping the other basis elements to zero. It is easy to see that the only cases for which $D_{a}(f_{\alpha}(b \wedge c))$ is non-zero is if $a=x_i$ or $y_i$, and $b$ and $c$ are $x_j$ and $y_j$, for some $i$ and $j$. It follows that $\lambda(D_{x_i})$ equals $\sum_{j=1}^m f_{x_i \wedge x_j \wedge y_j}$, and $\lambda(D_{y_i})$ equals $\sum_{j=1}^m f_{x_j \wedge y_j \wedge y_i}$.

\section{Huebschmann's Lemma for Lie Algebras}\label{huebschmann_section}
The main goal of this section is to prove Lemma \ref{Huebschmann} that gives a characterisation of crossed modules of Lie algebras whose equivalence class is trivial. Huebschmann has given such a characterisation in the case of groups in \cite[Section 10]{hueb80-1}. If $M$ is a $\Lg / \Lh$-module, one can alternatively deduce this characterisation from \cite{ratc80-1} or \cite{loda78-1}, once one verifies that Loday's map going into the relative cohomology group coincides with Ratcliffe's map. The approach in \cite{loda78-1} has been translated to the Lie algebra case by Kassel and Loday in \cite{kl82-1}, and can be adapted to the Lie algebra case if $M$ is a $\Lg / \Lh$-module. 
However, in the general case, we have to adapt Huebschmann's approach to give a formulation of the theorem in the setting of Lie algebras. 
We believe that homotopy of crossed $n$-fold extensions for Lie algebras (Subsection \ref{homotopy}) has not been defined before. Note that, though we introduce crossed $n$-fold extensions and complexes in general, we only need the case $n=2$ to prove the main result. 

\subsection{Crossed $n$-fold extensions}
Let $k$ be a ring. All Lie algebras will be modules over $k$. We use the classical concept of crossed modules to define (free) crossed complexes and (free) crossed $n$-fold extensions of Lie algebras. These definitions have been introduced by Huebschmann in \cite{hueb80-1} for the group case. We also mention some of their properties that will be needed later. 

A \emph{crossed complex} is a sequence 
\begin{equation}\label{crossed_complex}\underline{e}: \xymatrix{\cdots \ar[r] & C_3 \ar[r]^{\delta_3} & C_2 \ar[r]^{\delta_2} & \mathfrak{c}_1 \ar[r]^{\delta_1} & \mathfrak{g} }\end{equation}
with the following properties:
\begin{itemize}
	\item $\delta \circ \delta=0$;
	\item $\delta_1: \mathfrak{c}_1 \rightarrow \mathfrak{g}$ is a crossed module of Lie algebras (see Section \ref{third_cohom});
	\item for $i \geq 2$, $C_i$ is a $\mathfrak{q}$-module and $\delta_i$ is a $\mathfrak{q}$-module morphism, where $\mathfrak{q}$ is the cokernel of $\delta_1$. (Observe that $\Image \delta_2 \subseteq \Ker \delta_1$ is indeed a $\mathfrak{q}$-module.)
\end{itemize}
We will sometimes call this a crossed complex of $\mathfrak{q}$. If the sequence is exact, we call $\underline{e}$ a \emph{crossed resolution} of $\mathfrak{q}$. 
%n-fold extension
A \emph{crossed $n$-fold extension} of $\mathfrak{q}$ by a $\mathfrak{q}$-module $M$ is an exact sequence 
\begin{equation}\label{crossed_n_fold_extension}\xymatrix{0 \ar[r] &  M \ar[r]^-{i} & C_{n-1} \ar[r] & \cdots \ar[r] & C_2 \ar[r]^{\delta_2} & \mathfrak{c}_1 \ar[r]^{\delta_1} & \mathfrak{g} \ar[r]^{\pi} & \mathfrak{q} \ar[r] & 0},\end{equation}
where $\delta_1 : \mathfrak{c}_1 \rightarrow \mathfrak{g}$ is a crossed module,  $C_i$ is a $\mathfrak{q}$-module for $2 \leq i \leq n-1$ and the maps $\delta_i$ are $\Lq$-module homomorphisms for $2 \leq i \leq n-1$. This is a special case of a crossed resolution of $\mathfrak{q}$, and conversely, every crossed resolution of $\mathfrak{q}$ gives rise to a crossed $n$-fold extension, by taking $M = \Ker \delta_{n-1}$. A crossed 2-fold extension will simply be called a crossed extension. 
%
%morphisms
A \emph{morphism of crossed complexes} from $\underline{e}$ to $\underline{e}'$ is a sequence $\alpha=(\cdots, \alpha_2, \alpha_1, \alpha_0)$ that fits in the commutative diagram
\[\xymatrix{\underline{e} : & \cdots \ar[r] & C_2 \ar[d]^{\alpha_2} \ar[r]^{\delta_2} & \mathfrak{c}_1 \ar[d]^{\alpha_1} \ar[r]^{\delta_1} & \mathfrak{g} \ar[d]^{\alpha_0}\\
\underline{e}': & \cdots \ar[r] & C_2' \ar[r]^{\delta_2'} & \mathfrak{c}_1' \ar[r]^{\delta_1'} & \mathfrak{g}', }\]
such that $(\alpha_1, \alpha_0)$ is a morphism of crossed modules and for $i \geq 2$, $\alpha_i$ is a morphism of $\mathfrak{q}$-modules. Similarly, one defines morphisms $\alpha=(\alpha_n, \cdots, \alpha_1, \alpha_0)$ of crossed $n$-fold extensions.

Now we want to introduce the notion of free crossed complexes. 
By definition (see \cite{elli93-1}), a crossed module $\delta : \mathfrak{c} \rightarrow \mathfrak{g}$ is \emph{free} on a set map $\xi : X \rightarrow \mathfrak{g}$ if $X$ is a subset of $\mathfrak{c}$ such that the restriction of $\delta$ to $X$ equals $\xi$, and $\delta : \Lc \to \Lg$ satisfies the following universal property: for every crossed module $\delta':\mathfrak{c}' \rightarrow \mathfrak{g}$ and every set map $f : X \rightarrow \mathfrak{c}'$ with $\delta ' \circ f = \xi$, the map $f$ extends uniquely to a map $\widetilde{f} : \mathfrak{c} \rightarrow \mathfrak{c}'$ for which $(\widetilde{f}, \mathbb{1}_{\mathfrak{g}})$ is a morphism of crossed modules. The set $X$ together with the set map $\xi$ is called the \emph{basis} for the crossed module. It follows immediately from the definition that two free crossed modules with the same basis are isomorphic as crossed modules. 
If $\mathfrak{g}$ is a free Lie algebra, we call $\delta : \mathfrak{c} \rightarrow \mathfrak{g}$ \emph{totally free}. 

The construction of a free crossed module on a set map $\xi:X \rightarrow \mathfrak{g}$ is as follows (see \cite{elli93-1}). Let $k_X$ be the free $k$-module on $X$, and let $L$ be the induced $\mathfrak{g}$-module $L=U\mathfrak{g} \otimes_k k_X$. 
If we consider the adjoint action of $\mathfrak{g}$ on itself, there is a unique extension of $\xi$ to 
a $\mathfrak{g}$-module map $\xi': L \rightarrow \mathfrak{g}$. 
Let $\mathfrak{l}$ be the free Lie algebra on $L$. Using the universal property of the free Lie algebra, we see that there is a unique Lie algebra homomorphism $\delta_{\xi} : \mathfrak{l} \rightarrow \mathfrak{g}$, extending $\xi'$. 
Furthermore, there is a unique way to extend the $\mathfrak{g}$-action on $L$ to an action of $\mathfrak{g}$ on $\mathfrak{l}$ by derivations (see e.g.~\cite[p.~60]{baht87-1} or use an argument similar to the semi-direct product argument in the first part of the proof of Lemma \ref{up_to_homotopy}). Observe that, with this action, $\delta_{\xi}$ becomes a $\mathfrak{g}$-module morphism, since the adjoint action also acts by derivations. It is clear that the inclusion $L \hookrightarrow \mathfrak{l}$ preserves the action.
Finally, we take the quotient $\mathfrak{c}_{\xi}$ of $\mathfrak{l}$ by the ideal generated by the \emph{Peiffer elements} $^{\delta_{\xi}(x)}y-[x,y]$, $x, \, y \in \mathfrak{l}$. Since the $\mathfrak{g}$-action on $\mathfrak{l}$ takes Peiffer elements to Peiffer elements, there is an induced $\mathfrak{g}$-action by derivations on $\mathfrak{c}_{\xi}$.
It is clear that $\delta_{\xi} : \mathfrak{c}_{\xi} \rightarrow \mathfrak{g}$ is indeed a crossed module. We verify the universal property.
\begin{lemma}\label{free_crossed}
The crossed module $\delta_{\xi}: c_{\xi} \rightarrow \mathfrak{g}$ is the free crossed module on $\xi$.
\end{lemma}
\begin{proof}
Let $\delta' : \mathfrak{c}' \rightarrow \mathfrak{g}$ be a crossed module and let $f : X \rightarrow \mathfrak{c}'$ be a set map with $\delta' \circ f = \xi$. We have to prove that there is a unique morphism of Lie algebras $\widetilde{f}: \Lc_{\xi} \rightarrow \Lc'$, extending $f$ and yielding a morphism of crossed modules $(\widetilde{f},\mathbb{1}_{\mathfrak{g}})$. The proof follows the same steps as the construction of the free crossed module. 

Observe that there is a unique linear map $k_X \rightarrow \mathfrak{c}'$ extending $f$. 
Since $L$ is an induced $\mathfrak{g}$-module, there is a unique linear map $f':L \rightarrow \mathfrak{c}'$ extending $f$ and preserving the $\mathfrak{g}$-action. 
This map extends to a $\mathfrak{g}$-equivariant Lie algebra morphism $\mathfrak{l} \rightarrow \mathfrak{c}'$, that factors through the quotient map $\mathfrak{l} \rightarrow \mathfrak{c}_{\xi}$ since $\mathfrak{c}'$ is a crossed module. This yields the required map $\widetilde{f}$. Uniqueness of the extension at every step guarantees that $\delta ' \circ \widetilde{f}=\delta_{\xi}$. 

To prove uniqueness of $\widetilde{f}$, suppose that there is another map $\phi : \mathfrak{c}_{\xi} \rightarrow \mathfrak{c}'$ extending $f$, and making $(\phi,\mathbb{1}_{\mathfrak{g}})$ into a morphism of crossed modules, or equivalently, $\delta' \circ \phi = \delta_{\xi}$ and $\phi$ preserves the action. 
Observe that the composition $\phi': \mathfrak{l} \rightarrow \mathfrak{c}_{\xi} \overset{\phi}{\rightarrow} \mathfrak{c}'$ also preserves the $\mathfrak{g}$-action. 
It is clear that, on $k_X$, $\phi'$ has to coincide with the linear extension of $f$. 
On the other hand, since $L \rightarrow \mathfrak{l}$ is $\mathfrak{g}$-equivariant, the restriction of $\phi'$ to $L$ will still preserve the $\mathfrak{g}$-action. 
By the universal property of the induced module, this shows that $\phi'$ and $f'$ coincide on $L$. 
By the universal property of the free Lie algebra, it follows that the maps coincide on $\mathfrak{l}$, and therefore $\phi=\widetilde{f}$. 
\end{proof}

One can prove the following lemma (compare in \cite[Lemma 2]{whit49-1}).
\begin{lemma}\label{lemma_free_crossed_module}
Given a free crossed module $\delta_{\xi}: \mathfrak{c}_{\xi} \rightarrow \mathfrak{g}$ on $\xi : X \rightarrow \mathfrak{g}$. Suppose $\delta' : \mathfrak{c}' \rightarrow \mathfrak{g}'$ is another crossed module, and $\phi:\mathfrak{g} \rightarrow \mathfrak{g}'$ is a Lie algebra morphism. If $f : X \rightarrow \mathfrak{c}'$ is a set map with $\delta' \circ f=\phi \circ \xi$, there is a unique map $\widetilde{f} : \mathfrak{c}_{\xi} \rightarrow \mathfrak{c}'$ extending $f$, such that $(\widetilde{f},\phi)$ is a morphism of crossed modules.  
\end{lemma}
The proof is analogous to the proof of Lemma \ref{free_crossed}, and is left to the reader. 

A \emph{free crossed complex (resolution)} is a crossed complex (resolution) of the form (\ref{crossed_complex}) where $\mathfrak{g}$ is a free Lie algebra, $\delta_1: \mathfrak{c}_1 \rightarrow \mathfrak{g}$ is a free crossed module (so it is totally free) and $C_i$ is a free $\mathfrak{q}$-module for $i \geq 2$. 
%Existence of free crossed resolutions for $\mathfrak{q}$
Fix an arbitrary Lie algebra $\mathfrak{q}$. We show that we can always construct a free crossed resolution of $\mathfrak{q}$ (compare \cite[Proposition 2]{hueb80-1}). Choose a generating set $X$ of $\mathfrak{q}$, and let $\mathfrak{g}$ be the free Lie algebra on $X$, with projection map $\pi : \mathfrak{g} \rightarrow \mathfrak{q}$. Choose a set $R \subseteq \mathfrak{g}$ that generates $\Ker \pi$ as an ideal of $\mathfrak{g}$. If $\xi : R \rightarrow \mathfrak{g}$ is the injection map, one can check that the sequence
\[\xymatrix{\mathfrak{c}_{\xi} \ar[r]^{\delta_{\xi}} & \mathfrak{g} \ar[r]^{\pi} & \mathfrak{q} \ar[r] & 0}\]
is exact, and $\delta_{\xi} : \mathfrak{c}_{\xi} \rightarrow \mathfrak{g}$ is totally free. Now take a free $\mathfrak{q}$-module $C_2$ mapping onto $\Ker \delta_{\xi}$, and define $\delta_2$ as the composition $C_2 \rightarrow \Ker \delta_{\xi} \rightarrow \mathfrak{c}_{\xi}$. This yields again an exact sequence. If we continue this way, we get a free crossed resolution
\begin{equation}\label{free_crossed}\underline{C}: \xymatrix{\cdots \ar[r] & C_3 \ar[r]^{\delta_3} & C_2 \ar[r]^{\delta_2} & \mathfrak{c}_{\xi} \ar[r]^{\delta_{\xi}} & \mathfrak{g} }\end{equation} of $\mathfrak{q}$.
In the same spirit, a \emph{free crossed $n$-fold extension} of $\mathfrak{q}$ by $M$ is a crossed $n$-fold extension of the form (\ref{crossed_n_fold_extension}) where $\delta_1 : \mathfrak{c}_1 \rightarrow \mathfrak{g}$ is totally free, and for $2 \leq i \leq n-1$, $C_i$ is a free $\mathfrak{q}$-module. It is obvious that a free crossed resolution (\ref{free_crossed}) of $\mathfrak{q}$ yields a free crossed $n$-fold extension of $\mathfrak{q}$ by $J_n$,
\[\underline{C}^n : \xymatrix{0 \ar[r] & J_n \ar[r] & C_{n-1} \ar[r]^-{\delta_{n-1}} & C_{n-2} \ar[r] & \cdots \ar[r] & \mathfrak{c}_{\xi} \ar[r]^{\delta_{\xi}} & \mathfrak{g} \ar[r]^{\pi} & \mathfrak{q} \ar[r] & 0,}\]  where $J_n$ is defined as $\Ker \delta_{n-1}$.

It is a well-known result that, in the classical setting, there always exists a map from a free complex to a resolution, lifting a given map on the right hand side. In the crossed setting, we have a similar result (compare \cite[Proposition 3]{hueb80-1}).
\begin{lemma} \label{lemma_free_crossed_complex}
Let $\underline{e}$ be a free crossed complex of $\mathfrak{q}$ and let $\underline{e}'$ be a crossed resolution of $\mathfrak{q}'$. Now any Lie algebra morphism $\phi : \mathfrak{q} \rightarrow \mathfrak{q}'$ can be lifted to a morphism of crossed complexes $\alpha: \underline{e} \rightarrow \underline{e}'$ (so $\phi \circ \delta_1=\pi' \circ \alpha_0$). 
\end{lemma}
To prove this lemma, one uses the fact that $\mathfrak{g}$ is a free Lie algebra, together with Lemma \ref{lemma_free_crossed_module} and the result in the classical case. 
It is immediate that this implies that there always exists a morphism from the free crossed $n$-fold extension $\underline{C}^n$ to any other crossed $n$-fold extension, lifting a given map $\phi : \mathfrak{q} \rightarrow \mathfrak{q}'$ (compare \cite[Proposition 3$'$]{hueb80-1}).

\subsection{Homotopy}\label{homotopy}
%What is homotopy? 
Let $\underline{e}$ and $\underline{e}'$ be two crossed complexes and suppose we are given two morphisms of crossed complexes $\alpha=(\cdots, \alpha_2, \alpha_1, \alpha_0, \psi)$ and $\beta=(\cdots, \beta_2, \beta_1, \beta_0, \psi)$ that coincide on $\mathfrak{q}$. 
\[\xymatrix{\underline{e} {\phantom{'}}: \  \cdots \ar[r] & C_2 \ar@<-0.5ex>[d]_{\alpha_2} \ar@<0.5ex>[d]^{\beta_2} \ar[r]^{\delta_2} & \mathfrak{c}_{1} \ar@<-0.5ex>[d]_{\alpha_1} \ar@<0.5ex>[d]^{\beta_1} \ar[r]^{\delta_1} & \mathfrak{g} \ar@<-0.5ex>[d]_{\alpha_0} \ar@<0.5ex>[d]^{\beta_0} \ar[r]^{\pi} & \mathfrak{q}\ar[d]^{\psi} \ar[r] & 0\\\underline{e}':  \ \cdots \ar[r] & C_2' \ar[r]^{\delta_2'} & \mathfrak{c}_{1}' \ar[r]^{\delta_1'} & \mathfrak{g}' \ar[r]^{\pi'} & \mathfrak{q}' \ar[r] & 0}\] A \emph{homotopy from $\alpha$ to $\beta$} is a sequence of maps $\Sigma=(\cdots, \Sigma_2, \Sigma_1, \Sigma_0)$, with $\Sigma_0: \mathfrak{g} \rightarrow \mathfrak{c}_{1}'$, $\Sigma_1: \mathfrak{c}_1 \rightarrow C_2'$ and $\Sigma_i : C_i \rightarrow C_{i+1}'$ for $i \geq 2$, such that the following properties hold:
\begin{itemize}
	\item the map $\Sigma_0$ is a linear map that satisfies the relations $\Sigma_0[x,y]={}^{\alpha_0(x)}\Sigma_0(y)-{}^{\beta_0(y)}\Sigma_0(x)$ and $\delta_1' \circ \Sigma_0 = \alpha_0 - \beta_0$;
	\item the map $\Sigma_1$ is a linear map that satisfies the relations $\Sigma_1(^{x}y)={}^{\pi'\alpha_0 (x)}\Sigma_1(y)$ ($={}^{\psi \pi (x)}\Sigma_1(y)$), $\Sigma_1[x,y]=0$ and $\delta_2' \circ \Sigma_1 + \Sigma_0 \circ \delta_1 = \alpha_1 - \beta_1$;
	\item for $i \geq 2$, $\Sigma_i$ is a morphism of $\mathfrak{q}$-modules and $\delta_{i+1}' \circ \Sigma_i + \Sigma_{i-1} \circ \delta_i=\alpha_i - \beta_i$.
\end{itemize}   Observe that, if $\delta_1 ' \circ \Sigma_0=\alpha_0 - \beta_0$, the fact that $\Sigma_0[x,y]={}^{\alpha_0(x)}\Sigma_0(y)-{}^{\beta_0(y)}\Sigma_0(x)$ is equivalent to  $\Sigma_0[x,y]=[\Sigma_0(x),\Sigma_0(y)]+^{\beta_0(x)}\Sigma_0(y)-{}^{\beta_0(y)}\Sigma_0(x)$, so we can give an equivalent definition of homotopy accordingly.
(Compare to the definition of homotopy in \cite[Section 6]{hueb80-1}.)
\begin{lemma}
The relation ``being homotopic'' is an equivalence relation. 
\end{lemma}
\begin{proof}
%Fix two crossed complexes $\underline{e}$ and $\underline{e}'$ like above. 
For $\alpha=\beta$, we take $\Sigma_i \equiv 0$ for $i \geq 0$. This is indeed a homotopy from $\alpha$ to itself. To prove symmetry, we take a homotopy $\Sigma=(\cdots, \Sigma_1, \Sigma_0)$ from $\alpha$ to $\beta$, and show that $-\Sigma=(\cdots, -\Sigma_1, -\Sigma_0)$ is a homotopy from $\beta$ to $\alpha$. 
The only difficulty lies in the behaviour of $-\Sigma_0$ on Lie brackets. 
Since $\delta_1' \circ \Sigma_0 = \alpha_0 - \beta_0$ and ${}^{\delta_1'(x)}y=[x,y]$ for $x, \, y \in \mathfrak{c}_1'$, one can see
that ${}^{\alpha_0(x)}\Sigma_0(y)={}^{\beta_0(x)}\Sigma_0(y)+[\Sigma_0(x),\Sigma_0(y)]$ and ${}^{\beta_0(y)}\Sigma_0(x)={}^{\alpha_0(y)}\Sigma_0(x)-[\Sigma_0(y),\Sigma_0(x)]$. 
It is now easy to show that $-\Sigma$ is a homotopy from $\beta$ to $\alpha$. 
Furthermore, suppose that $\Sigma$ is a homotopy from $\alpha$ to $\beta$ and $\Sigma'$ is a homotopy from $\beta$ to $\gamma$. We claim that the sum $\Sigma+\Sigma'=(\cdots, \Sigma_1 + \Sigma_1', \Sigma_0 + \Sigma_0')$ is a homotopy from $\alpha$ to $\gamma$. The second and third conditions are easily verified. 
 Using the fact that $\delta_1' \circ \Sigma_0 = \alpha_0 - \beta_0$ and $\delta_1' \circ \Sigma_0' = \beta_0 - \gamma_0$, it is easy to verify that ${}^{\alpha_0(x)}\Sigma'_0(y)-{}^{\gamma_0(y)}\Sigma_0(x)={}^{\beta_0(x)}\Sigma_0'(y)-{}^{\beta_0(y)}\Sigma_0(x)$. The remainder of the proof is straight-forward. 
\end{proof}
Now the following lemma makes it possible to require uniqueness up to homotopy of the morphism in Lemma \ref{lemma_free_crossed_complex}  (compare \cite[Proposition 4]{hueb80-1}).

\begin{lemma}\label{up_to_homotopy}
If $\underline{e}$ is a free crossed complex of $\mathfrak{q}$, $\underline{e}'$ is a crossed resolution of $\mathfrak{q}'$, and $\alpha$ and $\beta$ are two morphisms $\underline{e} \rightarrow \underline{e}'$ extending $\psi:\mathfrak{q} \rightarrow \mathfrak{q}'$, then $\alpha$ and $\beta$ are homotopic. 
\end{lemma}
\begin{proof}
%sigma0
Let $\underline{e}$ be a free crossed complex and $\underline{e}'$ a crossed resolution, of the same form as before. Take a generating set $X$ for the free Lie algebra $\mathfrak{g}$. Since $\pi' \circ \alpha_0=\psi \circ \pi = \pi' \circ \beta_0$, we can choose for every $x \in X$ an element $\Sigma(x) \in \mathfrak{c}_1'$ such that $\delta_1' \circ \Sigma(x)=\alpha_0(x)-\beta_0(x)$. 
If we give $\mathfrak{c}_1'$ a $\mathfrak{g}$-module structure through $\beta_0$, we can define a set map $X \rightarrow \mathfrak{c}'_1 \rtimes \Lg$, mapping $x$ to $(\Sigma(x),x)$. 
This gives rise to a Lie algebra morphism $\Lg \rightarrow \Lc_1' \rtimes \Lg$, and we define $\Sigma_0$ as the first component of this map. One can easily check that $\Sigma_0[x,y]=[\Sigma_0(x),\Sigma_0(y)]+^{\beta_0(x)}\Sigma_0(y)-{}^{\beta_0(y)}\Sigma_0(x)$. To show that $\delta_1' \circ \Sigma_0 = \alpha_0 - \beta_0$, one proceeds by induction on the brackets.
 
%sigma1
Since $\Lc_1$ is the free crossed module on a set $R \subseteq \Lg$, we can construct a map $\Sigma_1 : \Lc_1 \rightarrow C_2'$ as follows. It is clear that we can choose a map $f: R \rightarrow C_2'$, such that for every $r \in R$, $\delta_2' \circ f(r)=\alpha_1(r) - \beta_1 (r)- \Sigma_0 \circ \delta_1(r)$. 
Now, observe that $\alpha_1- \beta_1 - \Sigma_0 \circ \delta_1$ is trivial on brackets, which follows from the fact that $\Image \delta_2' \subseteq Z(\Lc_1')$,  and that the map takes an element ${}^{x}y$ to ${}^{\alpha_0(x)}(\alpha_1(y)-\beta_1(y) - \Sigma_0 \circ \delta_1(y))$. Take the crossed module $0: C_2' \rightarrow \Lq'$, $C_2'$ being an abelian Lie algebra. If we look at the Lie algebra map $\psi \circ \pi : \Lg \rightarrow \Lq'$, Lemma \ref{lemma_free_crossed_module} shows that $f$ extends to a map $\Sigma_1 : \Lc_1 \rightarrow C_2'$, such that $(\Sigma_1,\psi \circ \pi)$ is a morphism of crossed modules. This means in particular that $\Sigma_1[x,y]=0$ and $\Sigma_1({}^{x}y)={}^{\psi \pi (x)}\Sigma_1(y)$. 
%%%%%%%%%%%%%%
Let $\Ll$ and $L$ be as in the construction of a free crossed module. Since $\delta_2 ' \circ \Sigma_1$ and $\alpha_1 - \beta_1- \Sigma_0 \circ \delta_1$ are both linear maps that behave similarly with respect to the action and coincide on $R$, the equality $\delta_2 ' \circ \Sigma_1 + \Sigma_0 \circ \delta_1 = \alpha_1 - \beta_1$ holds for elements in the image of $L \rightarrow \mathfrak{l} \rightarrow \Lc_1$. On the other hand, both $\delta_2 ' \circ \Sigma_1$ and $\alpha_1 - \beta_1- \Sigma_0 \circ \delta_1$ are zero on the brackets, so the equality $\delta_2 ' \circ \Sigma_1 + \Sigma_0 \circ \delta_1 = \alpha_1 - \beta_1$ holds for elements in the image of $\mathfrak{l} \rightarrow \Lc_1$. This map is surjective, so we conclude that $\delta_2 ' \circ \Sigma_1 + \Sigma_0 \circ \delta_1= \alpha_1 - \beta_1$.

%sigma i 
The construction of $\Sigma_i$ for $i \geq 2$ is classical.
\end{proof}

Of course, one can define a similar homotopy relation for morphisms $(\alpha_n, \cdots, \alpha_0, \psi)$ and $(\beta_n, \cdots, \beta_0,\psi)$ of crossed $n$-fold extensions, by taking $\Sigma_i=0$ for $i \geq n$. %, and using the definition for homotopy of crossed complexes. 
Observe that here, $\Sigma_{n-1} \circ \delta_n = \alpha_n - \beta_n$. In this case, the previous lemma becomes as follows.

\begin{lemma}\label{lemma_uniqueness_homotopy}
If $\underline{e}$ is a free crossed $n$-fold extension of $\mathfrak{q}$, $\underline{e}'$ is a crossed $n$-fold extension of $\mathfrak{q}'$, and $\alpha$ and $\beta$ are two morphisms $\underline{e} \rightarrow \underline{e}'$ extending $\psi:\mathfrak{q} \rightarrow \mathfrak{q}'$, then $\alpha$ and $\beta$ are homotopic. 
\end{lemma}

With this notion of homotopic maps, one can introduce homotopy equivalence between crossed complexes (crossed $n$-fold extensions): two crossed complexes (crossed $n$-fold extensions) $\underline{e}$ and $\underline{e}'$ are \emph{homotopy equivalent} if there exist morphisms $\alpha: \underline{e} \rightarrow \underline{e}'$ and $\beta : \underline{e}' \rightarrow \underline{e}$ such that $\alpha \circ \beta$ is homotopic to $\mathbb{1}_{\underline{e}'}$ and $\beta \circ \alpha$ is homotopic to $\mathbb{1}_{\underline{e}}$. The next lemma follows readily from the above.

\begin{lemma}
Any two free crossed resolutions (free crossed $n$-fold extensions) of $\mathfrak{q}$ are homotopy equivalent. 
\end{lemma}
\subsection{Trivial crossed extensions}\label{trivial_crossed_modules}
In the previous subsections, we have introduced some of the components we need to prove the main result of this section, Lemma \ref{Huebschmann}, that gives a  characterisation of crossed extensions that are ``trivial'' in some sense. The proof of the theorem is a translation of section 8 and the proof in section 10 of \cite{hueb80-1} to the case of Lie algebras. We write it down here for completeness. 

%equivalence of crossed n-fold extensions
We can define a relation ``$\sim$'' on the crossed $n$-fold extensions as follows: $\underline{e} \sim \underline{e}'$ iff there is a morphism of crossed $n$-fold extentions $(\mathbb{1}, \alpha_{n-1}, \cdots, \alpha_1, \alpha_0, \mathbb{1}): \underline{e} \rightarrow \underline{e}'$. This relation generates an equivalence relation, and the equivalence classes of crossed $n$-fold extensions of $\mathfrak{q}$ by $M$ are noted as $\Opext^n(\mathfrak{q},M)$.

%Zero extension
Now we turn to the case $n=2$. Define
\[\underline{O}:\xymatrix{0 \ar[r] & M \ar@{=}[r] & M \ar[r]^{0} & \mathfrak{q} \ar@{=}[r] & \mathfrak{q} \ar[r] & 0}\]
as the \emph{zero crossed 2-fold extension} of $\mathfrak{q}$ by $M$. Under the well-known bijection $\Opext^2 (\mathfrak{q},M) \cong H^3(\mathfrak{q},M)$, for which an explicit proof can be found in \cite{wage06-1}, the equivalence class of $\underline{O}$ clearly corresponds to the trivial cohomology class.
Lemma \ref{Huebschmann} gives a characterisation of all the crossed 2-fold extensions that are equivalent with $\underline{O}$, or, equivalently, that correspond to the trivial cohomology class under the bijection $\Opext^2 (\mathfrak{q},M) \cong H^3(\mathfrak{q},M)$. We repeat the statement here, renaming the Lie algebras and maps.

\begin{lemma} %\label{Huebschmann}
A crossed extension 
\[\underline{e}:\xymatrix{0 \ar[r] & M \ar[r]^{i} & \mathfrak{c} \ar[r]^{\delta} & \mathfrak{g} \ar[r]^{\rho} & \mathfrak{q} \ar[r] & 0}\]
is equivalent to the zero extension if and only if there exists  a short exact sequence \break $\underline{e}': \xymatrix@1{0 \ar[r]& \mathfrak{c} \ar[r] & \mathfrak{e} \ar[r]& \mathfrak{q} \ar[r]& 0}$ of Lie algebras and a Lie algebra homomorphism $h : \mathfrak{e} \rightarrow \mathfrak{g}$ such that the diagram
\[\xymatrix{ \underline{e}': {\phantom{0}} & 0 \ar[r]& \mathfrak{c} \ar@{=}[d]\ar[r]& \mathfrak{e} \ar[d]^{h}\ar[r]& \mathfrak{q} \ar@{=}[d] \ar[r]& 0\\
\underline{e}{\phantom{'}} :  0 \ar[r]& M \ar[r]^{i}& \mathfrak{c} \ar[r]^{\delta} & \mathfrak{g} \ar[r]^{\rho} & \mathfrak{q} \ar[r] & 0}\]
commutes and $(\mathbb{1}_{\mathfrak{c}},h)$ is a homomorphism of crossed modules.
\end{lemma}
\begin{proof}
If such a diagram exists, it is not difficult to see that the extension is equivalent to the zero extension, for example by constructing the associated 3-cocycle.

Now suppose that $\underline{e}$ is equivalent to the zero crossed extension $\underline{O}$.
Take a free crossed 2-fold extension $\underline{C}^2$ of $\mathfrak{q}$. We know that there exists a morphism $(\nu,\beta_1,\beta_0,\mathbb{1}_{\Lq})$,
\[\xymatrix{\underline{C}^2:  0 \ar[r] & J_2 \ar[d]^{\nu} \ar[r]^{j} & \mathfrak{c}_1 \ar[r] \ar[d]^{\beta_1}& \mathfrak{f} \ar[r]^{\pi} \ar[d]^{\beta_0} & \mathfrak{q} \ar@{=}[d]\ar[r] & 0\\
\underline{e} :  0 \ar[r] & M \ar[r]^{i} & \mathfrak{c} \ar[r] & \mathfrak{g} \ar[r]^-{\rho} & \mathfrak{q} \ar[r] & 0,}\] that is unique up to homotopy. We claim that
we can extend $\nu$ to an $\mathfrak{f}$-module morphism $\widetilde{\nu}: \mathfrak{c}_1 \rightarrow M$, where the $\mathfrak{f}$-module structure on $M$ is given through $\pi: \mathfrak{f} \rightarrow \mathfrak{q}$.
Observe that there exists a sequence of crossed extensions $\underline{e}_i$, $0 \leq i \leq n$, with $\underline{e}_0=\underline{O}$, $\underline{e}_n=\underline{e}$, and morphisms of crossed extensions $\underline{e}_0 \rightarrow \underline{e}_{1}$,  $\underline{e}_2 \rightarrow \underline{e}_{1}$,  $\underline{e}_2 \rightarrow \underline{e}_{3}$, etc. (If needed, one can take $\underline{e}_1=\underline{e}_0$.) 
It is clear that, if $\underline{e}=\underline{O}$, we can take $\widetilde{\nu}=\beta_1 : \mathfrak{c}_1 \rightarrow M$, since $\beta_0$ coincides with $\pi$. 
%e_i \rightarrow e_{i+1}
Furthermore, suppose that there exists a diagram with commuting squares
\[\xymatrix{\underline{C}^2:  0 \ar[r] & J_2 \ar[d]^{\mu} \ar[r]^{j} & \mathfrak{c}_1 \ar[dl]^{\widetilde{\mu}}\ar[r] \ar[d]^{\beta_1'}& \mathfrak{f} \ar[r]^{\pi} \ar[d]^{\beta_0'} & \mathfrak{q} \ar@{=}[d]\ar[r] & 0\\
\underline{e}':  0 \ar[r] & M \ar@{=}[d] \ar[r] & \mathfrak{c}' \ar[d]^{\alpha_1} \ar[r] & \mathfrak{g}' \ar[d]^{\alpha_0} \ar[r] & \mathfrak{q} \ar@{=}[d]\ar[r] & 0\\
 \underline{e} : 0 \ar[r] & M \ar[r] & \mathfrak{c} \ar[r] & \mathfrak{g} \ar[r] & \mathfrak{q} \ar[r] & 0,}\] where $\widetilde{\mu}$ is an $\mathfrak{f}$-module morphism that extends $\mu$. It follows from Lemma \ref{lemma_uniqueness_homotopy} that the morphisms $(\nu, \beta_1, \beta_0, \mathbb{1}_{\Lq})$ and $(\mu, \alpha_1 \circ \beta_1', \alpha_0 \circ \beta_0', \mathbb{1}_{\Lq})$  are homotopic, say through a homotopy $(\Sigma_1, \Sigma_0)$. Since $\nu=\mu+(\Sigma_1 \circ j)$, we can take $\widetilde{\nu}=\widetilde{\mu}+\Sigma_1$. It is easy to see that indeed $\widetilde{\nu} \circ j = \nu$ and that $\widetilde{\nu}$ preserves the $\mathfrak{f}$-action.
%e_{i} \leftarrow e_{i+1}
On the other hand, if there is a diagram with commuting squares
\[\xymatrix{\underline{C}^2:  0 \ar[r] & J_2 \ar[d]^{\mu} \ar[r]^-{j} & \mathfrak{c}_1 \ar[dl]^{\widetilde{\mu}}\ar[r] \ar[d]^{\beta_1'}& \mathfrak{f} \ar[r]^{\pi} \ar[d]^{\beta_0'} & \mathfrak{q} \ar@{=}[d]\ar[r] & 0\\
\underline{e}':  0 \ar[r] & M  \ar[r] & \mathfrak{c}'  \ar[r] & \mathfrak{g}'  \ar[r] & \mathfrak{q} \ar[r] & 0\\
 \underline{e} : 0 \ar[r] & M \ar@{=}[u] \ar[r] & \mathfrak{c} \ar[u]_{\alpha_1} \ar[r] & \mathfrak{g} \ar[u]_{\alpha_0} \ar[r] & \mathfrak{q} \ar@{=}[u]\ar[r] & 0,}\]
where $\widetilde{\mu}$ is again an $\mathfrak{f}$-module morphism that extends $\mu$, the morphisms $(\nu, \alpha_1 \circ \beta_1, \alpha_0 \circ \beta_0, \mathbb{1}_{\Lq})$ and $(\mu, \beta_1', \beta_0',\mathbb{1}_{\Lq})$ are homotopic, say through $(\Sigma_1, \Sigma_0)$. This means that again, $\nu=\mu + (\Sigma_1 \circ j)$ and we can take $\widetilde{\nu}=\widetilde{\mu}+\Sigma_1$ as before. By induction, these two cases prove the claim.

Define $\widetilde{\beta}_1(x)=\beta_1(x)-i \circ \widetilde{\nu}$. It is straight-forward that $\widetilde{\beta}_1(j(J_2))=0$, so it induces a map $\beta: \mathfrak{n}=\Ker \pi \rightarrow \mathfrak{c}$. Using the properties of crossed modules, one can see that $\widetilde{\nu}[x,y]=0$ for all $x, \, y \in \mathfrak{c}$ and that $i(M)$ is central in $\mathfrak{c}$. 
It follows that $\beta$ is a morphism of Lie algebras and $(\beta, \beta_0) : (\mathfrak{n}, \mathfrak{f},\iota) \rightarrow (\mathfrak{c}, \mathfrak{g}, \delta)$ is a morphism of the crossed modules, where $\iota$ is the inclusion $\mathfrak{n} \rightarrow \mathfrak{f}$. 
We take the co-equalizer $\mathfrak{e}$ of the maps $\beta \rtimes 0 : \mathfrak{n} \rightarrow \mathfrak{c} \rtimes \mathfrak{f}$ and $0 \rtimes \iota : \mathfrak{n} \rightarrow \mathfrak{c} \rtimes \mathfrak{f}$, where the $\mathfrak{f}$-action on $\mathfrak{c}$ is defined via $\beta_0$. 
Explicitly, $\mathfrak{e}$ is the quotient of $\mathfrak{c} \rtimes \mathfrak{f}$ by the ideal consisting of all $(\beta(x),-\iota(x))$ with $x \in \mathfrak{n}$. 
By the universal property of the co-equalizer (or by checking explicitly), the map $\mathfrak{c} \rtimes \mathfrak{f} \rightarrow \mathfrak{q}$, $(x,y) \mapsto \pi(x)$, induces a surjective map $\overline{\pi}: \mathfrak{e} \rightarrow \mathfrak{q}$.
This map yields an exact sequence $\xymatrix@1{0 \ar[r]& \mathfrak{c} \ar[r]^-{\overline{i}} & \mathfrak{e} \ar[r]^{\overline{\pi}} & \mathfrak{q} \ar[r] & 1}$, where $\overline{i}$ is the composition $\mathfrak{c} \hookrightarrow \mathfrak{c} \rtimes \mathfrak{f} \twoheadrightarrow \mathfrak{e}$. 
One can check that there is a well-defined Lie algebra homomorphism $\alpha: \mathfrak{e} \rightarrow \mathfrak{g}$, 
induced by $\mathfrak{c} \rtimes \mathfrak{f} \rightarrow \mathfrak{g}$, $(x,y) \mapsto \delta(x)+\beta_0(y)$, 
making $(\mathbb{1}_{\Lc},\alpha):(\mathfrak{c} , \mathfrak{e} , \overline{i}) \rightarrow  (\mathfrak{c},\mathfrak{g},\alpha)$ into a morphism of crossed modules. 
Furthermore, the diagram 
\[\xymatrix{ & 0 \ar[r] & \mathfrak{c} \ar@{=}[d]\ar[r]^{\overline{i}} & \mathfrak{e} \ar[d]^-{\alpha}\ar[r]^{\overline{\pi}} & \mathfrak{q} \ar@{=}[d]\ar[r] & 0\\
0 \ar[r] & M \ar[r] & \mathfrak{c} \ar[r]^{\delta} & \mathfrak{g} \ar[r] & \mathfrak{q} \ar[r] & 0}\] commutes, which finishes the proof. 
\end{proof}

\bibliography{G:/algebra/ref}

\begin{thebibliography}{10}

\bibitem{baht87-1}
Bahturin, Y.~A.
\newblock {\em Identical relations in {L}ie algebras}.
\newblock VNU Science Press b.v., Utrecht, 1987.
\newblock Translated from the Russian by Bakhturin.

\bibitem{cw81-1}
Cheng, C.~C. and Wu, Y.~C.
\newblock {\em An eight-term exact sequence associated with a group extension}.
\newblock Michigan Math. J., 1981, 28 3, 323--340.

\bibitem{dhw11-1}
Dekimpe, K., Hartl, M., and Wauters, S.
\newblock A seven-term exact sequence for the cohomology of a group extension.
\newblock To appear in Journal of Algebra.

\bibitem{elli93-1}
Ellis, G.~J.
\newblock {\em Homotopical aspects of {L}ie algebras}.
\newblock J. Austral. Math. Soc. Ser. A, 1993, 54 3, 393--419.

\bibitem{gers64-1}
Gerstenhaber, M.
\newblock {\em A uniform cohomology theory for algebras}.
\newblock Proc. Nat. Acad. Sci. U.S.A., 1964, 51 626--629.

\bibitem{gers66-1}
Gerstenhaber, M.
\newblock {\em On the deformation of rings and algebras. {II}}.
\newblock Ann. of Math., 1966, 84 1--19.

\bibitem{hs53-2}
Hochschild, G. and Serre, J.-P.
\newblock {\em Cohomology of {L}ie algebras}.
\newblock Ann. of Math. (2), 1953, 57 591--603.

\bibitem{hueb80-1}
Huebschmann, J.
\newblock {\em Crossed {$n$}-fold extensions of groups and cohomology}.
\newblock Comment. Math. Helv., 1980, 55 2, 302--313.

\bibitem{kl82-1}
Kassel, C. and Loday, J.-L.
\newblock {\em Extensions centrales d'alg\`ebres de {L}ie}.
\newblock Ann. Inst. Fourier (Grenoble), 1982, 32 4, 119--142 (1983).

\bibitem{loda78-1}
Loday, J.-L.
\newblock {\em Cohomologie et groupe de {S}teinberg relatifs}.
\newblock J. Algebra, 1978, 54 1, 178--202.

\bibitem{macl71-1}
MacLane, S.
\newblock {\em Categories for the working mathematician}.
\newblock Springer-Verlag, New York, 1971.
\newblock Graduate Texts in Mathematics, Vol. 5.

\bibitem{ratc80-1}
Ratcliffe, J.~G.
\newblock {\em Crossed extensions}.
\newblock Trans. Amer. Math. Soc., 1980, 257 1, 73--89.

\bibitem{sant83-1}
Santharoubane, L.~J.
\newblock {\em Cohomology of {H}eisenberg {L}ie algebras}.
\newblock Proc. Amer. Math. Soc., 1983, 87 1, 23--28.

\bibitem{wage06-1}
Wagemann, F.
\newblock {\em On {L}ie algebra crossed modules}.
\newblock Comm. Algebra, 2006, 34 5, 1699--1722.

\bibitem{weib94-1}
Weibel, C.~A.
\newblock {\em An introduction to homological algebra}, volume~38 of {\em
  Cambridge Studies in Advanced Mathematics}.
\newblock Cambridge University Press, Cambridge, 1994.

\bibitem{whit49-1}
Whitehead, J. H.~C.
\newblock {\em Combinatorial homotopy. {II}}.
\newblock Bull. Amer. Math. Soc., 1949, 55 453--496.

\end{thebibliography}
\bibliographystyle{G:/algebra/ref}

\end{document}